\documentclass[10pt]{amsart}
\usepackage[utf8]{inputenc} 
\usepackage[T1]{fontenc}
\usepackage{amsfonts,amscd}
\usepackage[a4paper, margin=1.25in]{geometry}
\usepackage[parfill]{parskip}
\usepackage{xcolor}
\usepackage{comment}
\usepackage{amssymb,latexsym,amsmath,amsthm,enumitem,mathrsfs}
\usepackage{fullpage}
\usepackage{fancyhdr}
\usepackage{hyperref}
\usepackage{diagbox}
\usepackage{array}
\usepackage{tikz}
\usetikzlibrary{arrows}
\usetikzlibrary{positioning}
\usepackage{float}\usepackage{mathtools} 
\usepackage{dsfont}
\usepackage[titletoc,title]{appendix}

\newtheorem{theorem}{Theorem}
\newtheorem{thm}[theorem]{Theorem}  

\newtheorem{prop}[theorem]{Proposition}
\newtheorem{lemma}[theorem]{Lemma}
\newtheorem{remark}[theorem]{Remark}
\newtheorem{cor}[theorem]{Corollary}

\newtheorem{example}[theorem]{Example}

\mathtoolsset{showonlyrefs}

\def\Z{\mathbb Z}

\def\R{\mathbb R}

\allowdisplaybreaks
\begin{document}
\title{The metric theory of small gaps for a sequence of real numbers}
\author[Jewel]{Jewel Mahajan}
\address{Jewel Mahajan, National Institute of Science Education and Research (NISER) Bhubaneswar, Jatni, Padanpur, Odisha 752050}
\email{jewelmahajan421@gmail.com}

\begin{abstract} 
Let $(a_n)_{n \geq 1}$ be a sequence of distinct positive integers. The metric theory of minimal gaps for the sequence $\{\alpha a_n \text{ mod }1, 1\leq n \leq N\}$ as $N \to \infty$ was initiated by Rudnick~\cite{Rudnick}, who established that the minimal gap admits an asymptotic upper bound expressible in terms of the additive energy of $\{a_1,\ldots,a_N\}$ for almost every $\alpha$. Later, Aistleitner, El-Baz, and Munsch \cite{ABM} demonstrated that the metric theory of minimal gaps for such sequences is not governed by the additive energy, but by the cardinality of the difference set of $\{a_1,\ldots,a_N\}$. They established a sharp convergence test for the typical asymptotic order of the minimal gap and proved general upper and lower bounds that are readily applicable. A key element of their proof relies on the resolution of the Duffin--Schaeffer conjecture by Koukoulopoulos and Maynard.
In this article, we generalise several results from~\cite{ABM} on \emph{integer} sequences to the case of \emph{real} sequences. Although an upper bound for $\delta_{\min}^{\alpha}(N)$ remains elusive, we obtain one for its floored counterpart $\lfloor \delta^{\alpha}_{\min} \rfloor (N)$ (both defined in Section~\ref{Introduction}) for real sequences $(a_n)_{n \geq 1}$ of distinct numbers. Our theorems recover \cite[Theorems 1--3]{ABM}, as well as the result from \cite[Section 4.3]{ABM}. Furthermore, we establish lower bounds for the minimal gaps of well-spaced sequences and, more generally, of a broader family that contains them.   
\end{abstract}
\maketitle

\section{Introduction}\label{Introduction}
The theory of uniform distribution modulo one is a fundamental area at the intersection of number theory, ergodic theory, and metric Diophantine approximation. A central problem within this field is to understand the local statistics of the sequences $(\alpha a_n)_{n \geq 1}$, where $\alpha$ is a real parameter and $(a_n)_{n \geq 1}$ is a given sequence of distinct real numbers.
A key statistic capturing the finest scale behaviour is the \emph{minimal gap}, defined for $N$ terms as
\begin{align}
   \delta_{\min}^{\alpha}(N) = \min \{ \|\alpha(a_m - a_n)\| : 1 \leq m \neq n \leq N \},
\end{align}
where $\|x\|=\displaystyle\min_{k \in \Z} |x-k|$.
The behaviour of $\delta_{\min}^{\alpha}(N)$ as $N \to \infty$ for generic $\alpha$ (that is, for almost all $\alpha$ with respect to the Lebesgue measure) reveals the deep arithmetic properties of the sequence $(a_n)_{n \geq 1}$. As direct computation of the minimal gap $\delta_{\min}^{\alpha}(N)$ for fixed $\alpha$ is typically difficult, we analyse its metric properties, focusing on the almost-sure behaviour and distribution.

The expected behaviour of the minimal gap serves as a powerful tool for distinguishing random from deterministic sequences. The minimal gap for a random sequence is always at most $1/N$, as $1/N$ is the average spacing. For a sequence of $N$ independent and uniformly distributed random points in the unit interval, the minimal gap is almost surely of size about \(1/N^2\), a hallmark of Poissonian statistics where points exhibit no correlations. This random benchmark provides a crucial baseline: a deterministic sequence is considered to exhibit Poissonian behaviour if its minimal gap also decays as $1/N^2$ as in the random case. 

In contrast to the random case, more structured sequences, such as $( \sqrt{n})_{n \notin \square}$, are more rigidly spaced, resulting in a larger minimal gap that decays more slowly (see \cite{Regavim}), thus signifying a clear departure from random behaviour. The minimal gap in this case is known to be of the order $2/N^{3/2}$, which is much larger than the random scaling of $1/N^2$. Since polynomial sequences $(n^d)_{n \geq1}$ for  $d \geq 2$ and lacunary sequences have nearly minimal additive energy (close to  $N^2$), the minimal gap $\delta_{\min}^{\alpha}(N)$ is of order $1/N^2$ up to sub-power factors. This follows from the following results of Rudnick \cite[Theorems 1--2]{Rudnick}.

Let $(a_n)_{n \geq 1}$ be a sequence of distinct integers, and let its additive energy be nearly minimal, that is, $\ll N^{2+o(1)}$. Then, for almost all $\alpha$, any $\eta > 0$ , and for all sufficiently large  $N$, we have
\begin{align}
\frac{1}{N^{2+\eta}} < \delta_{\min}^{\alpha}(N) < \frac{1}{N^{2-\eta}}.    
\end{align}
The landmark work of Aistleitner, El-Baz, and Munsch \cite{ABM} established that the metric theory of minimal gaps for integer sequences is not controlled by the additive energy, but rather by the cardinality of the difference set $\{a_i -a_j: 1\leq i  \neq j \leq N\}$, providing a comprehensive framework that connects the almost-sure minimal gap of $(\alpha a_n)_{n \geq 1}$ directly to its difference set structure. This article significantly extends their line of research to sequences $(a_n)_{n \geq 1}$ of arbitrary distinct positive real numbers. We introduce and study several natural variants of the minimal gap $\delta_{\min}^{\alpha}(N)$, namely $\lfloor \delta_{\min}^{\alpha} \rfloor(N)$, $\widetilde{\delta}_{\min}^{\alpha}(N)$, and $\widehat{\delta}_{\min}^{\alpha}(N)$ (defined in Sections~\ref{Introduction} and~\ref{minimal gap statistics}). A primary aim is to understand the relationships between these quantities and to determine their typical asymptotic order for almost all $\alpha$.

Our main results are multifaceted. First, we establish a general lower bound and a weak upper bound for $\delta_{\min}^{\alpha}(N)$ under the mild assumption that the sequence $(a_n)_{n \geq 1}$ has a uniform positive lower bound on its gaps, $\displaystyle\inf_{m \neq n} |a_m - a_n| \geq c > 0$. These bounds are expressed in terms of $D_N $, defined as the number of distinct positive differences between $\{a_1,\ldots,a_N\}$, and feature precise logarithmic factors endemic to such metric problems.

Second, we prove a more refined theorem where the constant lower bound $c$ is replaced by a monotonically decreasing function $c(N)$, allowing the treatment of sequences whose terms can cluster. This provides a framework for handling a wider class of sequences, including polynomial sequences with exponents less than $1$.

Third, we establish several upper bounds for $\lfloor \delta_{\min}^{\alpha} \rfloor(N)$, where 
    \begin{align}
       \lfloor \delta_{\min}^{\alpha} \rfloor(N) = \min \left\{ \| \alpha \lfloor |a_m - a_n| \rfloor \| : 1 \leq m \neq n \leq N \right\},
    \end{align}
using the powerful machinery of the Duffin-Schaeffer conjecture, recently resolved by Koukoulopoulos and Maynard (see \cite{KoukouMaynard}). These include three types of bounds: one that holds for infinitely many $N$ and depends on $H_N$ (the cardinality of the set of floored positive differences) and the size of $a_N$, another that depends solely on $H_N$, and a third bound in terms of $H_N$ that is valid for all sufficiently large $N$.

Finally, for any given sequence $(\eta(N))_{N \geq 1}$ of non-negative reals, we present a complete zero-one law that characterises whether the inequality $\lfloor \delta_{\min}^{\alpha} \rfloor(N) \leq \eta(N)$ holds for infinitely many $N$ and for almost no or almost all $\alpha$. This theorem provides a unifying criterion that recovers and generalises previous results.

Our work demonstrates that the fundamental principles governing the fine-scale statistics of sequences, previously understood primarily in the integer setting, have powerful extensions to the realm of real sequences. The results have applications to lacunary sequences, Beatty sequences, and sequences defined by smooth, sufficiently monotonic functions, several examples of which are explored in detail in later sections.

The investigation of sequences modulo one sits at the rich confluence of number theory, ergodic theory, and probability. A sequence \((a_n)_{n \geq 1}\) in \([0,1)\) is considered uniformly distributed if, in the limit, the proportion of its terms that fall into any subinterval is equal to the length of the interval. This ``global'' equidistribution, a concept firmly established since the pioneering work of Weyl \cite{Weyl1916}, is elegantly quantified by discrepancy theory, which measures the maximum deviation from this ideal distribution \cite{MichaelTichy, kuipers1974uniform}.

However, global uniformity often masks a more complex microscopic reality. Two sequences can be perfectly uniformly distributed, yet exhibit radically different local behaviours. The true arithmetic character of a sequence is frequently revealed not by its global distribution but by its \emph{local statistics}, the fine-scale patterns such as the distribution of gaps between consecutive points and the correlation between the spacings of multiple points. The study of these statistics, including pair correlation and nearest-neighbour gap distributions, has proven to be a remarkably sensitive probe. It forges deep and often unexpected connections to diverse fields such as random matrix theory, quantum chaos, and the metric Diophantine approximation, providing a unifying language for describing randomness and structure in deterministic systems.

Although local statistics of integer sequences have been extensively studied (see \cite{ABM2021, Aistleitner-Larcher-Lewko, Bloom-Chow, Rudnick, Rudnick-Sarnak}), with important work on the distribution of spacings and pair correlation for sequences like $(n^2\alpha)_{n \geq 1}$ \cite{RudnickSarnakAlexandru2001} and other lacunary sequences \cite{RudnickAlexandru1999, RudnickAlexandru2002}, the natural and vast generalisation to sequences of arbitrary real numbers has remained largely unexplored. This gap is significant, as many sequences of central importance in analysis and mathematical physics---such as lacunary sequences, values of smooth functions, and geometric progressions with non-integer ratios---reside firmly within this real-valued realm. A comprehensive theory for real sequences is, therefore, not merely a technical extension but a necessary step for a complete understanding.

This article is motivated by the goal of establishing such a theory for the minimal gap statistic. We seek to generalise the seminal findings of Aistleitner, El-Baz, and Munsch \cite{ABM} for integer sequences, which tied the minimal gap to the cardinality of the difference set, into the domain of real sequences. By introducing and analysing new, natural variants of the minimal gap tailored for real numbers, we aim to uncover the fundamental principles that govern their fine-scale distribution. Our work provides a unified framework that not only recovers classical results but also opens the door to analysing a much broader and richer class of sequences, thereby deepening our understanding of the interplay between arithmetic structure and random-like behaviour.

\subsection{Organisation of the article}
The article is structured as follows. In Section~\ref{main results}, we summarise all main results of this article without proofs. We also highlight key corollaries and provide a preview of applications to various sequence families.

Section~\ref{minimal gap statistics} is dedicated to the study of minimal gap statistics for sequences of distinct positive real numbers. We begin by defining four key metrics: $\delta_{\min}^{\alpha}(N)$, $\widetilde{\delta}_{\min}^{\alpha}(N)$, $\widehat{\delta}_{\min}^{\alpha}(N)$, and $\lfloor \delta_{\min}^{\alpha} \rfloor(N)$. We then establish the fundamental chain of inequalities $\delta_{\min }^{\alpha}(N)\leq \widetilde{\delta}_{\min }^{\alpha}(N)\leq \widehat{\delta}_{\min }^{\alpha}(N)$. The core of this section is the proof of Theorem~\ref{thm:min_spacing}, which provides nearly optimal lower and upper bounds for $\delta_{\min}^{\alpha}(N)$ for sequences with uniform positive separation $\displaystyle\inf_{m\neq n} |a_m - a_n| \geq c > 0$. The proof employs a careful application of the Borel–Cantelli lemma. We also  survey relevant prior results, particularly those of Aistleitner, El-Baz, and Munsch for integer sequences (see \cite{ABM}).

We then generalise this result in Theorem~\ref{thm:main} to sequences where the minimal separation $\displaystyle\inf_{1\leq m \neq n \leq N} |a_m - a_n|$ is bounded below by a monotonically decreasing sequence. A novel and careful enumeration of the difference set is introduced to handle this more delicate situation. 
A key consequence, detailed in Corollary~\ref{cor:main}, recovers and extends known results, such as Rudnick's theorem, by showing that for sequences with a fixed positive separation, one has $\delta_{\min}^{\alpha}(N) \gg 1 / N^{2+\eta}$ for any $\eta > 0$ and for almost all $\alpha$.

In Section~\ref{Proof of Theorem 6.4 upper bound}, we establish several key upper bounds for the metric $\lfloor \delta_{\min}^{\alpha} \rfloor(N)$. First, we prove a general upper bound in Theorem~\ref{upper bound}. Subsequently, we refine this analysis to prove Theorem~\ref{upper bound independent of a_N}, which provides an upper bound for $\lfloor \delta_{\min}^{\alpha} \rfloor(N)$ that holds for infinitely many $N$ and, crucially, depends only on the counting function $H_N$. A central component of this proof is a direct application of the celebrated Koukoulopoulos--Maynard theorem (Theorem~\ref{Duffin-Schaeffer Conjecture}). Finally, in Section~\ref{Subsection:upper bound in terms of H_N for sufficiently large N}, we establish Theorem~\ref{upper bound in terms of H_N for sufficiently large N}, which gives an upper bound for $\lfloor \delta_{\min}^{\alpha} \rfloor(N)$ in terms of $H_N$ that is valid for all sufficiently large $N$.

In Section~\ref{Proof of theorem three}, we present a complete zero-one law characterisation for $\lfloor \delta_{\min}^{\alpha} \rfloor(N) \leq \eta(N)$ to occur infinitely often in Theorem~\ref{THM 3}. The proofs of Lemma~\ref{lemma 3}, Theorem~\ref{upper bound independent of a_N}, and Theorem~\ref{THM 3} build upon the corresponding arguments in \cite{ABM} for integer sequences, incorporating key modifications developed in previous lemmas to adapt to our setting. We provide detailed proofs for completeness

\section{Main Results}\label{main results}
Before stating the main results, we collect the key notation and definitions used throughout the article.
\subsection{Notation and definitions}\label{Notation}
\begin{itemize}
    \item $\lfloor x \rfloor$: The floor of a real number $x$, denoting the greatest integer less than or equal to $x$.
    \item $\{x\}$: The fractional part of a real number $x$, defined by $\{x\} = x - \lfloor x \rfloor$.
    \item $\|x\|$: The distance from $x$ to the nearest integer. For $x \in \mathbb{R}$, $\|x\| = \min(\{x\}, 1 - \{x\})$.
    \item $\log x$: $\max(1, \log x)$, where $\log$ is the natural logarithm.
    \item $\log_k x$: The $k$-th iterated logarithm, defined recursively by $\log_1 x = \log x$ and $\log_k x = \log(\log_{k-1} x)$ for $k \geq 2$. Thus, $1/\log n$, $ \log_{2}n$, $\log_{3}n$, etc. are positive and well-defined for all integers $n \geq 1$.
    \item $f(x) \ll g(x)$: $|f(x)| \leq C |g(x)|$ for some constant $C > 0$ and for all $x$ in a specified domain. Equivalently, $f(x) = O(g(x))$. We write $g(x) \gg f(x)$ to mean $f(x) \ll g(x)$.
    \item $f(x) = o(g(x))$: $\frac{f(x)}{g(x)} \to 0$ as $x \to \infty$.
    \item $\nu_p(n)$: The $p$-adic valuation of $n \in \mathbb{N}$, that is, the largest integer $k$ such that $p^k$ divides $n$.
    \item $A^+$: The set of all positive elements of a set $A$.
    \item $\lfloor A \rfloor$: For any subset $A \subseteq \mathbb{R}$, $\lfloor A \rfloor = \{\lfloor a \rfloor : a \in A\}$.
    \item $A_N$: The finite set of the first $N$ terms, $A_N = \{a_1, a_2, \dots, a_N\}$.
    \item $A_N - A_N$: The entire difference set, $\{a_m - a_n : 1 \leq m, n \leq N\}$.
    \item $(A_N - A_N)^+$: The set of positive differences, $\{|a_m - a_n| : 1 \leq m \neq n \leq N\}$.
    \item $\lfloor(A_N - A_N)^+ \rfloor$: The set of floored positive differences. 
    \item $ (\lfloor A_N - A_N \rfloor)^+$: The set of positive floored differences.
    \item $C_N$,\, $D_N$:  $C_N = \#(A_N - A_N)$, \quad $D_N = \#(A_N - A_N)^+$.
    \item $H_N$,\, $T_N$:  $H_N = \#\lfloor(A_N - A_N)^+ \rfloor$, \quad $T_N = \#(\lfloor A_N - A_N \rfloor)^+$.
    \item $\mathcal{N}(k)$: The first index $M$ at which the positive integer $k$ appears in the set of floored positive differences, defined by
    \begin{align}
    \mathcal{N}(k) = \min \left\{ M \geq 1 : k \in \lfloor (A_M - A_M)^+ \rfloor \right\}.
    \end{align}
    If $k \notin \cup_{N\geq 1}\lfloor (A_N - A_N)^+ \rfloor$, we set $\mathcal{N}(k) = \infty$ and adopt the convention that the supremum of the empty set is zero.
    \item Almost all $\alpha$: A property that holds for all $\alpha$ in a subset of $[0,1]$ with the Lebesgue measure $1$.
    \item $\mathds{1}_E(*)$: The indicator function of $E$, that is, $\mathds{1}_E(x)$ is equal to 1 if $x \in E$ and 0 otherwise. We adopt the convention that $\mathds{1}(x=0)$ denotes $\mathds{1}_{\{0\}}(x)$.

\end{itemize}
Equipped with this notation, we now state our main results. This work establishes new results on the fine-scale statistics of sequences $(\alpha a_n)_{n \geq 1}$ for generic $\alpha$. Our theorems provide comprehensive bounds, both lower and upper, that hold for infinitely many $N$ and for all sufficiently large $N$, along with a precise zero-one law for the minimal gap and its variants. The results reveal how the distribution of gaps depends on the arithmetic structure of the sequence through the cardinalities $D_N$, $H_N$, and related quantities.
\subsection{General lower and upper bounds}
Our first main result provides nearly optimal lower and upper bounds for the minimal gap $\delta_{\min}^{\alpha}(N)$ for sequences with a positive uniform separation. The limits are given in terms of the quantity $D_N $, the number of distinct positive differences between the first $N$ terms.

\begin{thm}\label{thm:min_spacing}
    Let $(a_{n})_{n \geq 1}$ be a sequence of distinct positive real numbers such that $\displaystyle\inf_{ m\neq n}|a_m -a_n| \geq c$ for some positive constant $c$. Let $\epsilon>0$, and let $D_N$ be defined as in Section~\ref{Notation}.  Then, for almost all $\alpha \in [0,1]$, and for all sufficiently large $N$ $($that is, there exists \( N_0(\alpha) \) such that for all \( N \geq N_0(\alpha) \)$)$, we have 
    
$\textbf{(a)}$ $$\delta_{\mathrm{min}}^{\alpha}(N)\geq \frac{1}{D_{N} \log N (\log_{2}N)^{1+\epsilon}},\text{ and}$$ 

$\textbf{(b)}$ $$\delta_{\mathrm{min}}^{\alpha}(N)\leq 1-\frac{1}{D_{N} \log N (\log_{2}N)^{1+\epsilon}}.$$
\end{thm}

\begin{remark}
  Since $D_N \leq C_N$, the statement of Theorem~\ref{thm:min_spacing} also holds when $D_N$ is replaced by $C_N$.
\end{remark} 

\begin{remark} 
In particular, condition $\displaystyle\inf_{m \neq n} |a_m - a_n| \geq 1$ is automatically satisfied for any sequence $(a_n)_{n \geq 1}$ of distinct positive integers. Therefore, we recover the second part of \cite[Theorem~1]{ABM}. 
\end{remark}

\begin{remark}
    For any $\epsilon >0$, we have $\frac{\log_{2}N}{\log{N}}+(1+\epsilon)\frac{\log_{3}N}{\log{N}}=o(1)$ as $N\to \infty$. Corresponding to a fixed $\eta>0$, there exists $N_0^{'}(\eta)$ such that $\frac{\log_{2}N+\log_{3}N^{(1+\epsilon)}}{\log{N}}<\eta $, that is, $ \log N\log_{2}N^{(1+\epsilon)}<N^{\eta }$, for all $N\geq N_{0}^{'}(\eta)$. Let $N_{0}^{'}(\alpha)$ be such that Theorem~\ref{thm:min_spacing} holds for all $N$ larger than $N_{0}^{'}(\alpha)$. Define $N_{0}(\alpha):= \max(N_{0}^{'}(\alpha),N_{0}^{'}(\eta))$. Since $ D_N \leq C_N \leq N^{2}$, we have 
    \begin{align}
         \delta_{\mathrm{min}}^{\alpha}(N)\geq \frac{1}{D_N \log N (\log_{2}N)^{1+\epsilon}}
         \geq \frac{1}{  N^{2}}\frac{1}{  N^{\eta}}= \frac{1}{  N^{2+\eta}}
    \end{align}
    for almost all $\alpha \in [0,1]$ and for every $ N \geq N_{0}(\alpha)$.
    
In other words, for each $\eta >0$ and for almost all $\alpha \in [0,1]$, we have $\delta_{\mathrm{min}}^{\alpha}(N)  \geq   \frac{1}{  N^{2+\eta}}$
    for every $ N \geq N_{0}(\alpha)$. Thus, we recover \cite[Theorem 1]{Rudnick} for a sequence of distinct real numbers.  
\end{remark}

\subsection{Refined bounds for sequences with decaying gaps}
We next present a refinement of Theorem~\ref{thm:min_spacing} that applies to sequences where the minimal gap between distinct terms is allowed to decay to zero at a controlled rate.

\begin{thm}\label{thm:main}
Let \((a_n)_{n \geq 1}\) be a sequence of distinct positive real numbers such that $\displaystyle\inf_{\substack{1\leq  m \neq n \leq N}} |a_m - a_n| \geq c(N)$,
where $(c(N))_{N\geq 1}$ is a monotone decreasing sequence of positive reals. Let $\epsilon>0$, and let $D_N$ be defined as in Section~\ref{Notation}. Then, for almost all \( \alpha \in [0,1] \), and for all sufficiently large $N$, we have

\textbf{(a)}  \(\delta^{\alpha}_{\min}(N) \geq \dfrac{c(N)}{D_N \log N (\log_2 N)^{1+\epsilon}}\), \text{ and} 

\textbf{(b)}  \(\delta^{\alpha}_{\min}(N) \leq 1 - \dfrac{c(N)}{D_N \log N (\log_2 N)^{1+\epsilon}}\).
\end{thm}

\subsection{Upper bound via the Duffin-Schaeffer conjecture}
For the quantity $\lfloor \delta^{\alpha}_{\min} \rfloor(N)$, we prove three types of upper bounds: first, a bound that holds for infinitely many $N$ and depends on the size of $a_N$; second, a bound that is independent of the size of $a_N$ and instead relies only on $H_N$; third, an upper bound for $\lfloor \delta_{\min}^{\alpha} \rfloor(N)$ in terms of $H_N$ that is valid for all sufficiently large $N$.
\begin{thm}\label{upper bound}
Let $(a_n)_{n \geq 1}$ be a sequence of distinct positive real numbers such that
$\displaystyle\inf_{\substack{1\leq  m \neq n \leq N}} |a_m - a_n| \geq 1$ for all $N \in \mathbb{N}$. Then, for almost all $\alpha \in [0,1]$, we have 
        \begin{align}
    \lfloor \delta^{\alpha}_{\min }\rfloor(N)  \leq  \frac{4\log _{2} B_{N}}{H_{N} \log (\frac{N}{9}) \log _{2} (\frac{N}{9})} \quad \text{for infinitely many }N,
    \end{align}
    where $ \lfloor \delta^{\alpha}_{\min }\rfloor(N) =\displaystyle \min_{1\leq m\neq n\leq  N}\| \alpha\lfloor |a_{m}-a_{n}| \rfloor\| $, $B_N= \displaystyle \max_{1\leq m\neq n\leq  N} \lfloor |a_{m}-a_{n}|\rfloor$ and $H_N =  \#\lfloor(A_N -A_N)^{+}\rfloor$. 
\end{thm}
\begin{remark}
        Since $B_N \leq \lfloor \max(a_1,\ldots, a_N)\rfloor$, the statement of Theorem~\ref{upper bound} also holds when $B_N$ is replaced by $\lfloor\max(a_1,\ldots, a_N)\rfloor$.
\end{remark}
\begin{remark}
For a sequence of distinct positive integers, we have  $\displaystyle\inf_{1\leq m \neq n\leq N}|a_{m}-a_{n}|\geq 1$, $ \lfloor \delta^{\alpha}_{\min }\rfloor(N)=\delta_{\min }^{\alpha}(N)$, and $D_N = H_N$. As a consequence, Theorem~\ref{upper bound} yields a result of Aistleitner, El-Baz, and Munsch \cite[Theorem 1 (Equation 5)]{ABM}, which corresponds to the integer case.
\end{remark}

\begin{thm}\label{upper bound independent of a_N}
    Let $(a_n)_{n \geq 1}$ be a sequence of distinct positive real numbers such that
$\displaystyle\inf_{\substack{1\leq  m \neq n \leq N}} |a_m - a_n| \geq 1$ for all $N \in \mathbb{N}$. Then, for almost all $\alpha \in [0,1]$, we have 
    \begin{align}
            \lfloor \delta^{\alpha}_{\min }\rfloor(N) \leq \frac{1}{H_{N}} \quad   \text{for infinitely many }N.
    \end{align}
\end{thm}
\begin{remark}
    Theorem~\ref{upper bound independent of a_N} finds a natural counterpart in the following result, which establishes a similar upper bound for $\lfloor \delta^{\alpha}_{\min }\rfloor(N)$, valid for all sufficiently large $N$.
\end{remark}
 \begin{thm}\label{upper bound in terms of H_N for sufficiently large N}
            Let $(a_n)_{n \geq 1}$ be a sequence of distinct positive real numbers such that
$\displaystyle\inf_{\substack{1\leq  m \neq n \leq N}} |a_m - a_n| \geq 1$ for all $N \in \mathbb{N}$. Then, for every $\epsilon >0$ and almost all $\alpha $, 
        $$ \lfloor \delta^{\alpha}_{\min }\rfloor(N) \leq \frac{N^{\epsilon}}{H_N} \quad \text{for all except finitely many }N. $$
        \end{thm}
\begin{remark}
    When $(a_n)_{n \geq 1}$ is an integer sequence, we have $ \lfloor \delta^{\alpha}_{\min }\rfloor(N)=\delta_{\min }^{\alpha}(N)$, and $D_N = H_N$. Thus, Theorems~\ref{upper bound independent of a_N} and~\ref{upper bound in terms of H_N for sufficiently large N} recover the results in \cite[Section 4.3]{ABM} and \cite[Theorem 2]{ABM}, respectively.
\end{remark}
\subsection{A zero-one law for the minimal gap}
Our final main result is a comprehensive zero-one law that completely characterises the set of $\alpha$ for which $\lfloor \delta^{\alpha}_{\min }\rfloor(N)$ is infinitely often smaller than a given threshold function $\eta(N)$.
\begin{thm}\label{THM 3}
Let $\left(a_{n}\right)_{n \geq 1}$ be a sequence of positive real numbers such that $\displaystyle\inf_{1\leq m \neq n\leq N}|a_{m}-a_{n}|\geq 1$. Let $(\eta(N))_{N \geq 1}$ be a sequence of non-negative reals. Let $\mathcal{A}$ denote the set of those $\alpha \in[0,1]$ for which $\lfloor \delta_{\min }^{\alpha} \rfloor (N) \leq \eta(N)$ holds for infinitely many $N$. Then, we have $\lambda(\mathcal{A})=0$ or $\lambda(\mathcal{A})=1$, depending on whether the series
\begin{align}\label{series conv div} 
\sum_{k=1}^{\infty} \varphi(k) \sup _{b \geq 1}\left\{\frac{\sup _{\ell \geq \mathcal{N}(b k)} \eta(\ell)}{b k}\right\} 
\end{align}
is convergent or divergent, respectively, where $\varphi$ is Euler's totient function and $\lambda$ denotes the Lebesgue measure on $\mathbb{R}$.
\end{thm}
\begin{remark}
Under the assumption that $\eta$ is decreasing, the series for which convergence or divergence is to be checked in Theorem~\ref{THM 3} takes the following simpler form
$$
\sum_{k=1}^{\infty} \varphi(k) \sup _{b \geq 1}\left\{\frac{\eta(\mathcal{N}(b k))}{b k}\right\}.$$ 
\end{remark}
\begin{remark}
    Theorem~\ref{THM 3} provides a sharp criterion reducing the problem of determining the almost sure behaviour of the minimal gap of $(\alpha a_n)_{n \geq 1}$ to the arithmetic structure of $\lfloor(A_N - A_N)^+ \rfloor$, showing that this floored difference set alone determines the typical asymptotic order, at least in principle. 
\end{remark}
\begin{remark}
    When $(a_n)_{n \geq 1}$ is a sequence of distinct positive integers, we have $\displaystyle\inf_{1\leq m \neq n\leq N}|a_{m}-a_{n}|\geq 1$ and $ \lfloor \delta^{\alpha}_{\min }\rfloor(N)=\delta_{\min }^{\alpha}(N)$. Thus, we recover \cite[Theorem $3$]{ABM} from Theorem~\ref{THM 3}.
\end{remark}

\subsection{Applications and examples}
To illustrate the applicability of our theorems, we present several classes of sequences for which our methods yield new estimates for $\delta_{\mathrm{min}}^{\alpha}(N)$.
\begin{itemize}
    \item \textbf{Polynomial sequences:} For $a_n = n^\theta$ with $\theta \geq 1$, we get $\delta_{\mathrm{min}}^{\alpha}(N) \gg (N^2 \log N (\log_2 N)^{1+\epsilon})^{-1}$ for any $ \epsilon >0$.
    \item \textbf{Polynomial sequences with exponent less than one:} For $a_n = n^\theta$ with $\theta < 1$, we obtain $\delta_{\mathrm{min}}^{\alpha}(N) \gg (N^{3-\theta} \log N (\log_2 N)^{1+\epsilon})^{-1}$ for any $ \epsilon >0$.
    \item \textbf{Lacunary sequences:} Sequences that satisfy $\frac{a_{n+1}}{a_n} \geq D > 1$, such as $a_n = e^n$ or $  \pi^n $.
    \item \textbf{Beatty sequences:} Sequences of the form $a_n = \lfloor n \theta \rfloor$ for irrational $\theta > 1$.
    \item \textbf{Sequences in arithmetic progressions:} Sequences of the form $a_n = a+(n-1)d$ where $a$ and $d$ are positive constants.
    \item \textbf{Sequences under smooth maps:} If $f$ is differentiable with $|f'|$ bounded away from zero, then $(f(a_n))_{n \geq 1}$ inherits the minimal gap properties of $(a_n)_{n \geq 1}$.
\end{itemize}
The subsequent sections of this article are devoted to the detailed proofs of the results mentioned in this section, followed by a discussion of their implications and connections to the existing literature.

\section{minimal gap statistics}\label{minimal gap statistics}
For a given $\alpha \in [0,1]$ and a sequence $(a_n)_{n \geq 1}$ of distinct real numbers, 
let $\delta_{\min }^{\alpha}(N)$ and $ \lfloor \delta^{\alpha}_{\min }\rfloor(N)$ be defined as in Section~\ref{Introduction}. We now introduce the following related variants of the minimal gap:

\begin{itemize}
     \item $\widetilde{\delta}_{\min}^{\alpha}(N)$: The minimal fractional part of scaled differences,
    \begin{align}
    \widetilde{\delta}_{\min}^{\alpha}(N) = \min \left\{ \{\alpha|a_m - a_n|\} : 1 \leq m \neq n \leq N \right\}.
    \end{align}
     \item $\widehat{\delta}_{\min}^{\alpha}(N)$: The minimal scaled absolute difference,
       \begin{align}
    \widehat{\delta}_{\min}^{\alpha}(N) = \min \left\{ \alpha |a_m - a_n| : 1 \leq m \neq n \leq N \right\}.
   \end{align}
    \end{itemize}
\begin{remark}
    The expression $\delta_{\min }^{\alpha}(N)$ depends on the sequence $(a_n)_{n \geq 1}$, as do the quantities $\widetilde{\delta}_{\min }^{\alpha}(N)$, $\widehat{\delta}_{\min }^{\alpha}(N)$, and $\lfloor \delta^{\alpha}_{\min }\rfloor(N)$, but we suppress this dependence in the notation for simplicity.
\end{remark}
Note that
\begin{align}
    \widetilde{\delta}_{\min }^{\alpha}(N)=&\min \left\{\{\alpha|a_{m}-a_{n}|\}: 1 \leq m \neq  n \leq N\right\}\\
    \leq &\{\alpha|a_{m'}-a_{n'}|\}\\
    \leq & \alpha|a_{m'}-a_{n'}| \quad \text{for all }m'\neq n'\leq N.
\end{align}
Now, taking the minimum over all $m' \neq n' \leq N$, we obtain $\widetilde{\delta}_{\min }^{\alpha}(N)\leq \widehat{\delta}_{\min }^{\alpha}(N)$.

Similarly,
\begin{align}
    \delta_{\min }^{\alpha}(N)=&\min \left\{  \|\alpha (a_{m}-a_{n})\|: 1 \leq m \neq  n \leq N\right\}\\
    \leq &\|\alpha (a_{m'}-a_{n'})\|\\
    = &\|\alpha |a_{m'}-a_{n'}|\|\\
    \leq &\{\alpha|a_{m'}-a_{n'}|\} \quad \text{for all }m'\neq n'\leq N.
\end{align}
Now, taking the minimum over all $m' \neq n' \leq N$, we obtain $\delta_{\min }^{\alpha}(N)\leq \widetilde{\delta}_{\min }^{\alpha}(N)$.
Therefore, for all $N \in \mathbb{N}$, we obtain
\begin{align}\label{metric comparison}
    \delta_{\min }^{\alpha}(N)\leq \widetilde{\delta}_{\min }^{\alpha}(N)\leq \widehat{\delta}_{\min }^{\alpha}(N).
\end{align}
\begin{remark}
It is evident from \eqref{metric comparison} that when studying the lower bounds for such quantities, it suffices to find a lower bound for $\delta_{\min }^{\alpha}(N)$ (which will provide a lower bound for the other two, possibly a weaker one), and to find an upper bound it suffices to find an upper bound for the quantity $\widehat{\delta}_{\min }^{\alpha}(N)$. Determining the precise growth rates of these three quantities poses interesting questions for further study.
\end{remark}
\begin{remark}
    If $(a_n)_{n \geq 1}$ is a sequence of positive integers, then we have $ \lfloor \delta^{\alpha}_{\min }\rfloor(N)=\delta_{\min }^{\alpha}(N)$.
\end{remark}

\subsection{\textbf{Proof of Theorem~\ref{thm:min_spacing}}}
We first arrange the elements of $A=\displaystyle\bigcup_{N \geq 1}\left(A_{N}-A_{N}\right)^{+}$ in a sequence. Let $\left(z_{n}\right)_{n \geq 1}$ be a sequence of distinct positive real numbers such that for all $N$, 
\begin{equation}\label{cardinality01}
\left\{z_{n}: 1 \leq n \leq D_{N}\right\}=\left(A_{N}-A_{N}\right)^{+} , 
\end{equation}
with \((0 =) D_1 < D_{2} <\ldots< D_N < D_{N+1} < \ldots\), and $ D_{N} \geq (N-1) $.

The sequence $\left(z_{n}\right)_{n \geq 1}$ need not be unique, as $D_{N+1}-D_{N}$ can be as large as $N$. However, the same proof works for any choice of $\left(z_{n}\right)_{n \geq 1}$ that satisfies \eqref{cardinality01}. Since $(a_n)_{n\geq 1}$ is a sequence of distinct positive reals satisfying $\displaystyle\inf_{m\neq n} | a_{m} -a_{n}| \geq c$, we have $z_n \geq c$ for all $n$,  and hence $\lfloor z_n\rfloor \geq 0$ for all $n$. 

For $n \geq 1$, we set
\begin{align}\label{psi definition}
    \psi(n)=\frac{1}{n(\log \sqrt{n})\left(\log _{2} \sqrt{n}\right)^{1+\varepsilon}},
\end{align}
 $$J_{n}(a)= \left(\frac{a}{z_{n}}-\frac{\psi(n)}{z_{n}}, \frac{a}{z_{n}}+\frac{\psi(n)}{z_{n}}\right),$$ 
and
\begin{align}\label{Sn definition}
S_{n}= [0,1] \cap \Big(\bigcup_{0 \leq a \leq \lfloor z_{n}\rfloor+1}J_{n}(a) \Big).
\end{align}

We have
\begin{align}
    \sum_{n=1}^{\infty} \lambda\left(S_{n}\right) 
    \leq  &\sum_{n=1}^{\infty} \min \left\{ \frac{2\psi(n)}{z_{n}} (\lfloor z_{n}\rfloor +2), 1\right\} \\
    \leq & \sum_{n=1}^{\infty} 2\psi(n) \left(1+\frac{2}{z_n} \right)\\
\leq & \sum_{n=1}^{\infty} 2\left(1+\frac{2}{c}\right)\psi(n) \leq  \sum_{n=1}^{\infty} \frac{2\left(1+\frac{2}{c}\right)}{n(\log \sqrt{n})\left(\log _{2} \sqrt{n}\right)^{1+\varepsilon}}<\infty.
\end{align}
Note that since $S_n\subseteq [0,1]$ we have $\lambda (S_n) = \text{Prob}(\{\alpha\in [0,1]:\alpha \in S_n\})$, and hence $$\displaystyle\sum_{n\geq 1} \text{Prob}(\{\alpha\in [0,1]:\alpha \in S_n\})< \infty.$$
By the first Borel-Cantelli lemma, since $\displaystyle\sum_{n=1}^\infty \lambda(S_n) < \infty$, we have
\begin{align}
0=&\text{Prob}(\{S_n \text{ infinitely often}\})=\text{Prob}\Big(\{\alpha\in [0,1]:\alpha \in \bigcap_{n=1}^{\infty} \bigcup_{i=n}^{\infty}S_i\} \Big).    
\end{align}
In other words, $ \text{Prob}\Big(\{\alpha\in [0,1]:\alpha \notin \displaystyle\bigcap_{n=1}^{\infty} \displaystyle\bigcup_{i=n}^{\infty}S_i\} \Big)=1$, that is, there exists a set $\tilde{Q} \subseteq [0,1]$ with $\lambda(\tilde{Q})=1$ such that for all $\alpha \in \tilde{Q}$, we have $\alpha \notin \displaystyle\bigcap_{n=1}^{\infty} \bigcup_{i=n}^{\infty}S_i$. This means that there exists $N_1(\alpha) \in \mathbb{N}$ such that $\alpha \notin \displaystyle\bigcup_{i=N_1(\alpha)}^{\infty}S_i$, that is, $\alpha \notin S_n$ for all $n \geq N_1(\alpha)$. 

Therefore, for all $n \geq N_1(\alpha)$, we have
\begin{align}
     & \alpha \notin [0,1] \cap \Big(\bigcup_{0 \leq a \leq \lfloor z_{n}\rfloor+1}J_{n}(a) \Big)\\
   \overset{\alpha \in [0,1]}{\implies}& \alpha \notin \bigcup_{0 \leq a \leq \lfloor z_{n}\rfloor+1}J_{n}(a) \\
    \implies & |\alpha -\frac{a}{z_{n}}| \geq \frac{\psi(n)}{z_{n}} \text{ for all } a \in \{0,1,\ldots, \lfloor z_{n}\rfloor+1\}\\
    \implies & |\alpha z_{n}-a| \geq \psi(n) \text{ for all } a \in \{0,1,\ldots, \lfloor z_{n}\rfloor+1\}.
\end{align}
Therefore, for all $\alpha \in \tilde{Q}$ and all $n \geq N_1(\alpha)$, we have
\begin{align}\label{minimum bound}
 & |\alpha z_{n}-a| \geq \psi(n) \text{ for all } a \in \{0,1,\ldots, \lfloor z_{n}\rfloor+1\}.
\end{align}

\textbf{Case 1: } Let $z_n \in [c,1)$. Then $\lfloor z_{n}\rfloor=0$. We have $0\leq \alpha z_n \leq z_n =\lfloor z_n\rfloor +\{z_n\} =\{z_n\}< 1$, so $|\alpha z_n -1|=1-\alpha z_n=1-\{\alpha z_n\}$. Therefore, using \eqref{minimum bound}, we obtain
\begin{align}\label{term 1}
 \| \alpha z_n\|=\min( \{\alpha z_{n}\},1-\{\alpha z_{n}\})
 =&\displaystyle\min_{a\in \{0,1\} }|\alpha z_{n} -a|
 = \displaystyle\min_{a\in \{0,\ldots,\lfloor z_{n}\rfloor+1\} }|\alpha z_{n} -a|
 \geq \psi(n).
\end{align}

\textbf{Case 2: } Let $z_n \in [1,\infty)$. Then, for all $\alpha \in [0,1]$, we have $0\leq \alpha z_n \leq z_n< \lfloor z_n\rfloor +1$.
Therefore, using \eqref{minimum bound}, we obtain
\begin{align}\label{term 2}
||\alpha z_{n}||=\displaystyle\min_{a\in \{0,1,\ldots,\lfloor z_{n}\rfloor+1\} }|\alpha z_{n} -a|\geq \psi(n).    
\end{align}

Therefore, combining \eqref{term 1} and \eqref{term 2}, we find that for all $\alpha \in \tilde{Q}$ and all $n \geq N_1(\alpha)$, we have $||\alpha z_{n}||\geq \psi(n)$. 

$\textbf{(a)}$ Let $a_{m}, a_{n} \in A_{N}$ with $m\neq n$. Then, $|a_{m}-a_{n}|=z_{n'}$ for some $n' =n'(m,n)\leq D_N$, where $D_N \leq N^{2}$. 

We now prove that for all sufficiently large $N$ and for any $\epsilon >0$, $\|\alpha z_{n'}\|  > \frac{1}{D_N \log N (\log_2 N)^{1+\epsilon}}$. We consider a finite set of indices $\{1, 2, \dots, N_1(\alpha)-1\}$. Since $z_n$ are positive and distinct, and since $\alpha$ is such that $\alpha z_n \notin \mathbb{Z}$ for any $n$ (which is true for almost all $\alpha$), there exists a constant $\eta(\alpha) > 0$ such that
\begin{align}
    \|\alpha z_n\| \geq \eta(\alpha) \quad \text{for all} \quad n < N_1(\alpha).
\end{align}
Furthermore, since $D_N \geq N-1 $, we have
\begin{align}
\frac{1}{D_N \log N (\log_2 N)^{1+\epsilon}} \to 0 \quad \text{as} \quad N \to \infty.
\end{align}
Therefore, there exists $N_2(\alpha)$ such that for all $N \geq N_2(\alpha)$,
\begin{align}
\frac{1}{D_N \log N (\log_2 N)^{1+\epsilon}} < \eta(\alpha).
\end{align}
Define $N_0(\alpha) = \max (N_1(\alpha), N_2(\alpha))$. Then, for all $N \geq N_0(\alpha)$ and every pair of distinct elements $a_m, a_n \in A_N$ with corresponding index $n'$:
\begin{itemize}
    \item If $n' \geq N_1(\alpha)$, then using the facts that $\psi$ is decreasing and $n' \leq D_N$, we have 
    \begin{align}
            \|\alpha z_{n'}\| \geq \psi(n') \geq \psi(D_N) =\frac{1}{D_N (\log \sqrt{D_N}) (\log_2 \sqrt{D_N})^{1+\epsilon}} \geq \frac{1}{D_N \log N (\log_2 N)^{1+\epsilon}},
    \end{align}
 where we used the facts that $D_{N} \leq C_{N} \leq N^{2}$.
    \item If $n' < N_1(\alpha)$, then
    \begin{align}
    \|\alpha z_{n'}\| \geq \eta(\alpha) > \frac{1}{D_N \log N (\log_2 N)^{1+\epsilon}}.
    \end{align}
\end{itemize}
Therefore, for all $N\geq N_0(\alpha)$, and all distinct $a_{m}, a_{n} \in A_{N}$, we have
\begin{align}
\left\|\alpha\left(a_{m}-a_{n}\right)\right\|=\left\|\alpha|a_{m}-a_{n}|\right\| 
=||\alpha z_{n'}||
\geq\frac{1}{D_N \log N (\log_2 N)^{1+\epsilon}}
\end{align}
for almost all $\alpha \in [0,1]$. The proof now follows by taking the minimum over all $a_{m}, a_{n}\in A_{N}$.

$\textbf{(b)}$  Since $0\leq \alpha z_{n} \leq z_{n}$, we have $0 \leq \lfloor \alpha z_{n} \rfloor  \leq \lfloor z_{n}\rfloor. $ In particular, taking $a=\lfloor \alpha z_n \rfloor $ in \eqref{minimum bound}, we have the following for all $\alpha \in \tilde{Q}$ (with $\lambda(\tilde{Q})=1$) and all $n \geq N_0(\alpha)$, 
\begin{align}\label{lower bound for fractional part}
 & \{\alpha z_{n}\} \geq \psi(n).
\end{align} 
Therefore, for all $N\geq N_0(\alpha)$ and all $a_{m}, a_{n} \in A_{N}$ with ${m}\neq {n} $, the inequality
\begin{align}
\left\{\alpha\left|a_{m}-a_{n}\right| \right\} 
\geq\frac{1}{D_N \log N (\log_2 N)^{1+\epsilon}}   
\end{align}
holds for almost all $\alpha \in [0,1]$.

Note that $||t||=\text{ min }(\{t\},1-\{t\}) \leq 1-\{t\}$ and $||-t||=||t||$ for any $t\in \mathbb{R}$. Therefore, using \eqref{lower bound for fractional part}, for all $N\geq N_0(\alpha)$ and all distinct $a_{m}, a_{n} \in A_{N}$, we have
\begin{align}
\left\|\alpha\left(a_{m}-a_{n}\right)\right\|=\left\|\alpha|a_{m}-a_{n}|\right\| \leq 1-\{\alpha|a_{m}-a_{n}|\}
\leq 1-\frac{1}{D_N \log N (\log_2 N)^{1+\epsilon}}   
\end{align}
for almost all $\alpha \in [0,1]$. The proof now follows by taking the minimum over all $a_{m}, a_{n}\in A_{N}$.

\begin{example}
    Let $\epsilon >0$, $m \in \mathbb{N}$, and  $d \geq 0$ be a fixed constant.  Then, Theorem~\ref{thm:min_spacing} implies that for any integer sequence $(a_n)_{n \geq 1}$, the sequence $(\frac{a_n}{m}+d)_{n \geq 1}$ satisfy 
 $$\delta_{\mathrm{min}}^{\alpha}(N)\geq \frac{1}{D_N \log N (\log_{2}N)^{1+\epsilon}}$$
for almost all $\alpha \in  [0, 1]$, and for all sufficiently large $N$, where $D_N$ is the cardinality of the set of positive differences for the sequence $(\frac{a_n}{m}+d)_{n \geq 1}$.
\end{example}
\begin{cor}\label{part 0}
Let $(a_{n})_{n \geq 1}$ be a sequence of distinct positive real numbers
such that $\displaystyle\inf_{ m\neq n}|a_m -a_n| \geq c$ for some constant $c>0$. Let $\epsilon>0$, and let $f$ be a differentiable function on $[\displaystyle\inf_{n}a_{n},\infty)$, which is positive-valued on $(a_n)_{n \geq 1}$ and for which there exists $\delta_{0} >0$ such that $|f'(x)|\geq \delta_0$ for all $x \in [\displaystyle\inf_{n}a_{n},\infty)$.
 Then, for almost all $\alpha \in [0,1]$ and for all sufficiently large $N$, we have 
 $$\delta_{\mathrm{min}}^{\alpha}((f(a_n))_{n\geq 1},N)\geq \frac{1}{\widetilde{D}_N \log N (\log_{2}N)^{1+\epsilon}}, \text{ and}$$
 $$\delta_{\mathrm{min}}^{\alpha}((f(a_n))_{n\geq 1},N)\leq 1-\frac{1}{\widetilde{D}_N \log N (\log_{2}N)^{1+\epsilon}},$$
 where $\widetilde{D}_N=\#(f(A_N)-f(A_N))^{+}$ and $A_N = \{a_1, a_2, \dots, a_N\}$.
\end{cor}
\begin{proof}
    By the mean value theorem, for any $m \neq n$, there exists $\xi \in [\inf (a_m,a_n),\sup (a_m,a_n)]$ such that 
    $$|f(a_m)-f(a_n)|=|f'(\xi)||a_m -a_n|\geq \delta_{0} |a_m -a_n|.$$

    Therefore, for all $m \neq n$, 
    \begin{align}
    |f(a_m)-f(a_n)|\geq \delta_{0} |a_m -a_n|\geq\delta_{0} \inf_{k\neq \ell} |a_{k} -a_{\ell}|\geq c \delta_{0}.
    \end{align}
    Hence, $\displaystyle\inf_{m\neq n}|f(a_m)-f(a_n)|\geq c \delta_{0} > 0$.
    
       Since $|f'(x)| \geq \delta_0 > 0$ for all $x$ and $f$ is differentiable, $f$ is injective. Thus, $(f(a_{n}))_{n \geq 1}$ is a sequence of distinct positive real numbers satisfying $\displaystyle\inf_{m\neq n}|f(a_m)-f(a_n)|\geq c \delta_{0} > 0$. We obtain the required bounds by applying Theorem~\ref{thm:min_spacing}.
\end{proof}

\begin{cor}\label{part 1}
Let $\epsilon>0$. Let $f$ be a differentiable function on $[1,\infty)$, which is positive-valued on $\mathbb{N}$, and for which there exists $\delta_{0} >0$ such that $|f'(x)|\geq \delta_0$ for all $x \in [1,\infty)$.
 Then, for almost all $\alpha \in [0,1]$ and for all sufficiently large $N$, we have 
 $$\delta_{\mathrm{min}}^{\alpha}((f(n))_{n \geq 1},N)\geq \frac{1}{\widetilde{D}_N \log N (\log_{2}N)^{1+\epsilon}}, \text{ and}$$
$$\delta_{\mathrm{min}}^{\alpha}((f(n))_{n \geq 1},N)\leq 1-\frac{1}{\widetilde{D}_N \log N (\log_{2}N)^{1+\epsilon}},$$
 where $\widetilde{D}_N=\#(f(A_N)-f(A_N))^{+}$ and $A_N = \{1, 2, \dots, N\}$.
\end{cor}
\begin{proof}
    For $a_n=n$, the set  $\{|a_{m}-a_{n}|:\, m\neq n\in \mathbb{N}\}$ is bounded below ($1$ is a lower bound) and has an infimum that satisfies
$\displaystyle\inf_{ m\neq n}|a_m -a_n| \geq c$ with $c=1$. The proof now follows from Corollary~\ref{part 0}.
\end{proof}

We now prove Theorem~\ref{thm:main}, a variant of Theorem~\ref{thm:min_spacing} in which the constant 
$c$ is replaced by a monotonically decreasing sequence $(c(N))_{N \geq 1}$.
\subsection{Proof of Theorem~\ref{thm:main}}
We first arrange the elements of $A=\displaystyle\bigcup_{N \geq 1}\left(A_{N}-A_{N}\right)^{+}$ in a sequence $\left(z_{n}\right)_{n \geq 1}$. To ensure a key monotonicity property, we choose this enumeration carefully.

For each $T \geq 2$, define the set of \emph{new} differences in step $T$ by
\begin{align}
B_T = (A_T - A_T)^+ \setminus (A_{T-1} - A_{T-1})^+.    
\end{align}
We now define the sequence \( (z_n)_{n \geq 1} \) by first listing all the elements of \( B_2 \) (in any order), then all the elements of \( B_3 \), and so on. This construction ensures that for all \( N \), we have
\begin{align}\label{cardinality}
\left\{z_{n}: 1 \leq n \leq D_{N}\right\}=\left(A_{N}-A_{N}\right)^{+},
\end{align}
with \((0 =) D_1 < D_{2} <\ldots< D_N < D_{N+1} < \ldots\). Observe that $D_N$ satisfies $ D_{N} \geq N-1$.

\noindent
For each \( n \in \mathbb{N}  \), we define \( f(n) \) as the smallest integer \( T \) such that \( z_n \in (A_T - A_T)^+ \). By our construction, if \( z_n \in B_T \), then \( f(n) = T \). Since the sequence \( (z_n)_{n \geq 1} \) is ordered by the sets \( B_2, B_3, \ldots \), the function \( n \mapsto f(n) \) is non-decreasing. Also, \( z_n \in (A_{f(n)} - A_{f(n)})^{+} \) and therefore 
\begin{equation}\label{z_n c(f(n)) relation}
z_n \geq \inf (A_{f(n)} - A_{f(n)})^{+}=\inf_{\substack{1\leq  r \neq s \leq f(n)}} |a_r - a_s| \geq  c(f(n)).  
\end{equation}

\noindent
For each \( n \in \mathbb{N} \), we set
\begin{align}
    \psi(n) = \frac{c(f(n))}{n (\log \sqrt{n}) (\log_2 \sqrt{n})^{1+\epsilon}}.
\end{align}
Note that $ (f(n))_{n\geq 1}$ is non-decreasing and $(c(n))_{n \geq 1}$ is decreasing, implying that $(c (f(n)))_{n\geq 1}$ is decreasing. Consequently, $(\psi(n))_{n \geq 1}$ is also decreasing. 

\noindent
For each \( n \in \mathbb{N}  \), we define
\begin{align}
  S_n = [0,1] \cap \bigcup_{a=0}^{\lfloor z_n \rfloor+ 1} \left( \frac{a}{z_n} - \frac{\psi(n)}{z_n}, \frac{a}{z_n} + \frac{\psi(n)}{z_n} \right).  
\end{align}
Using \eqref{z_n c(f(n)) relation}, we see that the Lebesgue measure of \( S_n \) satisfies
\begin{align}
\lambda(S_n) \leq (\lfloor z_n \rfloor + 2)  \frac{2 \psi(n)}{z_n} \leq 2 \psi(n) \left(1 + \frac{2}{z_n}\right) \leq 2 \psi(n) \left(1 + \frac{2}{c(f(n))}\right).
\end{align}
Substituting \(\psi(n)\), we obtain
\begin{align}
\lambda(S_n) \leq \frac{2 c(f(n))+4}{n (\log \sqrt{n}) (\log_2 \sqrt{n})^{1+\epsilon}} .
\end{align}
The sequence $(c(f(n)))_{n\geq1}$ being a monotone decreasing sequence of positive reals has an upper bound, namely $c(f(1))$. Therefore, summing over all \( n \in \mathbb{N}  \), we obtain
\begin{align}
\sum_{n=1}^\infty \lambda(S_n) \leq  \sum_{n=1}^\infty \frac{2c(f(1))+4}{n (\log \sqrt{n}) (\log_2 \sqrt{n})^{1+\epsilon}} < \infty,    
\end{align}
where convergence follows from the Cauchy condensation test. Therefore, by the Borel-Cantelli lemma, there exists a set $\tilde{Q} \subseteq [0,1]$ with $\lambda(\tilde{Q})=1$ (without loss of generality, we can assume that $\tilde{Q}\cap \mathbb{Q}=\phi$) such that for all $\alpha \in \tilde{Q}$ and all $n \geq N_1(\alpha)$, we have $\alpha \notin S_n$, that is,
\begin{align}\label{minimum bound for c(N) tending to 0}
 & |\alpha z_{n}-a| \geq \psi(n), \text{ for all } a \in \{0,1,\ldots, \lfloor z_{n}\rfloor+1\}.
\end{align}
Then, for all $\alpha \in [0,1]$, we have $0\leq \alpha z_n \leq z_n < \lfloor z_n\rfloor +1$.
Therefore, \eqref{minimum bound for c(N) tending to 0} ensures that for all $\alpha \in \tilde{Q}$ and for all $n \geq N_1(\alpha)$, we have 
\begin{align}
||\alpha z_{n}||=\displaystyle\min_{a\in \{0,1,\ldots,\lfloor z_{n}\rfloor+1\} }|\alpha z_{n} -a|\geq \psi(n).    
\end{align}

    $\textbf{(a)}$  Let $a_{m}, a_{n} \in A_{N}$ with $m\neq n$. Then, $|a_{m}-a_{n}|=z_{\tilde{n}}$ for some $\tilde{n} (=\tilde{n}(m,n))\leq D_N$, where $ D_N \leq N^{2}$. $f(D_N)$ is the smallest $T$ such that $z_{D_N} \in (A_T - A_T)^{+} $ and $z_{D_N}\in(A_N -A_N)^{+}$. Therefore, $f(D_N)\leq N$ for all $N$, and since $(c(n))_{n \geq 1}$ is decreasing, we have $c(f(D_N)) \geq c(N)$ for all $N$.

In the next few paragraphs, we prove $\|\alpha z_{\tilde{n}}\|  > \frac{c(N)}{D_N \log N (\log_2 N)^{1+\epsilon}}$ for all sufficiently large $N$. We consider a finite set of indices $\{1, 2, \dots, N_1(\alpha)-1\}$. Since $z_n$ are positive and distinct, and since $\alpha$ is such that $\alpha z_n$ is never an integer for any $n$ (which holds for almost all $\alpha$), there exists a constant $\eta(\alpha) > 0$ such that
\begin{align}
\|\alpha z_n\| \geq \eta(\alpha) \quad \text{for all} \quad n < N_1(\alpha).
\end{align}
Furthermore, $D_N \geq N-1 $ implies that the sequence
\begin{align}
\frac{c(N)}{D_N \log N (\log_2 N)^{1+\epsilon}} \to 0 \quad \text{as} \quad N \to \infty.
\end{align}
Therefore, there exists $N_2(\alpha)$ such that for all $N \geq N_2(\alpha)$,
\begin{align}
\frac{c(N)}{D_N \log N (\log_2 N)^{1+\epsilon}} < \eta(\alpha).
\end{align}
Now, let $N_0(\alpha) = \max\{N_1(\alpha), N_2(\alpha)\}$. Then, for all $N \geq N_0(\alpha)$ and for every pair of distinct elements $a_m, a_n \in A_N$ with the corresponding index $\tilde{n}$:
\begin{itemize}
    \item If $\tilde{n} \geq N_1(\alpha)$, then using the facts that $\psi$ is decreasing and $\tilde{n} \leq D_N$, we have 
    \begin{align}
    \|\alpha z_{\tilde{n}}\| \underbrace{\geq}_{\tilde{n} \geq N_1(\alpha)} \psi(\tilde{n}) \geq \psi(D_N) =\frac{c(f(D_N))}{D_N (\log \sqrt{D_N}) (\log_2 \sqrt{D_N})^{1+\epsilon}} \geq \frac{c(N)}{D_N \log N (\log_2 N)^{1+\epsilon}}.
    \end{align}
    \item If $\tilde{n} < N_1(\alpha)$, then
   \begin{align}
     \|\alpha z_{\tilde{n}}\| \underbrace{\geq}_{\tilde{n} < N_1(\alpha)} \eta(\alpha) \underbrace{>}_{N \geq N_{0}(\alpha)} \frac{c(N)}{D_N \log N (\log_2 N)^{1+\epsilon}}.
    \end{align}
\end{itemize}
Therefore, for all $N\geq N_0(\alpha)$, and for all distinct $a_{m}, a_{n} \in A_{N}$, we have
\begin{align}
\left\|\alpha\left(a_{m}-a_{n}\right)\right\|=\left\|\alpha|a_{m}-a_{n}|\right\| 
=||\alpha z_{\tilde{n}}||
\geq\frac{c(N)}{D_N \log N (\log_2 N)^{1+\epsilon}}
\end{align}
for almost all $\alpha \in [0,1]$. The proof now follows by taking the minimum over all $a_{m}, a_{n}\in A_{N}$.

$\textbf{(b)}$ Since $0\leq \alpha z_{n} \leq z_{n}$, we have $0 \leq \lfloor \alpha z_{n} \rfloor  \leq \lfloor z_{n}\rfloor. $ In particular, taking $a=\lfloor \alpha z_n \rfloor $ in \eqref{minimum bound for c(N) tending to 0}, we obtain that for all $\alpha \in \tilde{Q}$ (with $\lambda(\tilde{Q})=1$) and for all $n \geq N_0(\alpha)$, we have
\begin{align}\label{lower bound for fractional part00}
 & \{\alpha z_{n}\} \geq \psi(n).
\end{align} 
Therefore, for all $N\geq N_0(\alpha)$ and for all $a_{m}, a_{n} \in A_{N}$ with ${m}\neq {n} $, the inequality
\begin{align}
\left\{\alpha\left|a_{m}-a_{n}\right| \right\} 
\geq  \frac{c(N)}{D_N \log N (\log_2 N)^{1+\epsilon}}   
\end{align}
holds for almost all $\alpha \in [0,1]$.

Note that $||t||=\text{ min }(\{t\},1-\{t\}) \leq 1-\{t\}$ and $||-t||=||t||$ for any $t\in \mathbb{R}$. Therefore, using \eqref{lower bound for fractional part00}, for all $N\geq N_0(\alpha)$, and for all distinct $a_{m}, a_{n} \in A_{N}$, we have
\begin{align}
\left\|\alpha\left(a_{m}-a_{n}\right)\right\|=\left\|\alpha|a_{m}-a_{n}|\right\| \leq 1-\{\alpha|a_{m}-a_{n}|\}
\leq 1- \frac{c(N)}{D_N \log N (\log_2 N)^{1+\epsilon}}   
\end{align}
for almost all $\alpha \in [0,1]$. The proof now follows by taking the minimum over all $a_{m}, a_{n}\in A_{N}$.

\begin{cor}\label{cor:main}
Let \((a_n)_{n \geq 1}\) be a sequence of distinct positive real numbers such that $\displaystyle\inf_{\substack{1\leq  m \neq n \leq N}} |a_m - a_n| \geq c$, where $c$ is a positive constant. Let $\epsilon>0$, and let $D_N$ be defined as in Section~\ref{Notation}. Then, for almost all \( \alpha \in [0,1] \), there exists \( N_0(\alpha) \) such that for all \( N \geq N_0(\alpha) \), we have
\begin{align}
    \delta^{\alpha}_{\min}(N) \geq \dfrac{\max(1,c)}{D_N \log N (\log_2 N)^{1+\epsilon}}.
\end{align} 
\end{cor}
\begin{proof}
Taking $c(N)=c$ for all $N$, in Theorem~\ref{thm:main}, we obtain that for any \( \epsilon > 0 \) and for almost all \( \alpha \in [0,1] \),  \(\delta^{\alpha}_{\min}(N) \geq \dfrac{c}{D_N \log N (\log_2 N)^{1+\epsilon}}\) for all sufficiently large $N$. The proof now follows by combining this observation with Theorem~\ref{thm:min_spacing}. 
\end{proof}
\begin{remark}
    In particular, the Fibonacci sequence $(F_n)_{n \geq 2}$ (with $F_2 = 1, F_3 = 2, F_4 = 3, \dots$) satisfies $\inf_{m \neq n} |F_m - F_n| = 1$, and hence is a suitable choice for $a_n$ in Corollary~\ref{cor:main}.
\end{remark}
\begin{remark}[Beatty sequence]
    Let $a_n =\lfloor n\theta \rfloor$, where $\theta>1$ is an irrational number. Note that
$\lfloor a \rfloor - \lfloor b \rfloor =
\begin{cases}
\lfloor a - b \rfloor & \text{if } \{a\} \geq \{b\} \\
\lfloor a - b \rfloor + 1 & \text{if } \{a\} < \{b\}
\end{cases}$, showing that the difference $\lfloor (n+1)\theta \rfloor-\lfloor n\theta \rfloor $ can take only one of the two values $\lfloor \theta \rfloor$ or $\lfloor \theta \rfloor +1$. Therefore, for $m>n$, $|a_m -a_n| \in [(m-n)\lfloor \theta\rfloor, (m-n)(\lfloor \theta \rfloor+1) ]\cap \mathbb{N}$, that is, $\displaystyle\inf_{m\neq n} |a_m -a_n|\geq \lfloor \theta \rfloor\geq 1. $ Also, the inequality $\lfloor (n+1)\theta \rfloor -\lfloor n \theta \rfloor \geq \lfloor \theta  \rfloor \geq 1$ ensures that $(a_n)_{n \geq1}$ is a sequence of distinct positive integers. Therefore, applying Corollary~\ref{cor:main}, we find that for any \( \epsilon > 0 \) and for almost all \( \alpha \in [0,1] \), there exists \( N_0(\alpha) \) such that for all \( N \geq N_0(\alpha) \),
 $$\delta_{\mathrm{min}}^{\alpha}((\lfloor n\theta \rfloor)_{n\geq 1},N) \geq \dfrac{ \lfloor \theta \rfloor }{D_N \log N (\log_2 N)^{1+\epsilon}}\geq \dfrac{ \lfloor \theta \rfloor }{N^{2} \log N (\log_2 N)^{1+\epsilon}}.$$ 
In particular, one can consider the sequences $\lfloor \sqrt{2}n \rfloor $ and $\lfloor n \phi\rfloor$, where $\phi=\frac{1+\sqrt{5}}{2}$ is the golden ratio. 
\end{remark}
\begin{example}
    For positive constants $a$ and $d$, let $a_n =a+(n-1)d$. Then $ D_N = \#(A_N - A_N)^+=\# \{d,2d,\ldots, (N-1)d \}=N-1\leq N$, and we obtain
    $$\delta^{\alpha}_{\min}(N) \geq \frac{1}{D_N \log N (\log_2 N)^{1+\epsilon}}\geq \frac{1}{N \log N (\log_2 N)^{1+\epsilon}},$$ 
    which is $\gg \frac{1}{N (\log N)^{1+\epsilon}}$ for any $\epsilon >0$.
    
    In particular, taking $a=d=1$, that is, for the sequence $a_n =n$, we recover the result obtained in \cite[Section 3]{Rudnick}.
\end{example}

\begin{example}
Let $\theta \geq 1$ and $a_n = n^{\theta}$. Since $(a_n)_{n \geq 1}$ is a strictly increasing sequence and $f(x)=(x+1)^{\theta}-x^{\theta}$ is a monotone increasing function in $[1,\infty)$, we have
$|a_m -a_n| \geq |a_{n+1} -a_n| =(n+1)^{\theta}-n^{\theta} \geq 2^{\theta }-1$ for all $m>n$. This ensures that $\displaystyle\inf_{\substack{1\leq  m \neq n \leq N}} |a_m - a_n|=\displaystyle\inf_{\substack{1\leq  n<m\leq N}} |a_m - a_n| \geq 2^{\theta}-1$.
Therefore, using Corollary~\ref{cor:main}, we obtain that for almost all $\alpha \in [0,1]$, and for all sufficiently large $N$, 
 $$\delta_{\mathrm{min}}^{\alpha}((n^\theta)_{n\geq 1},N)\geq \frac{2^\theta-1}{D_N \log N (\log_{2}N)^{1+\epsilon}}\geq \frac{2^{\theta}-1}{N^{2} \log N (\log_{2}N)^{1+\epsilon}},\text{ and }$$
$$\delta_{\mathrm{min}}^{\alpha}((n^\theta)_{n\geq 1},N)\leq 1-\frac{2^{\theta}-1}{D_N \log N (\log_{2}N)^{1+\epsilon}}\leq 1-\frac{2^{\theta}-1}{N^{2} \log N (\log_{2}N)^{1+\epsilon}}, \text{ where}$$
 \( A_N = \{1^{\theta},2^{\theta}, \dots, N^{\theta}\} \) and \( D_N = \#(A_N - A_N)^+ \).
\end{example}

\begin{remark}
    The following functions can be taken in Corollary~\ref{part 1} as suitable choices for $f(x)\colon$ $\sinh (x)=\frac{e^{x}-e^{-x}}{2}$, $e^{x}$, $e^{x}+x^{k}$, $x^{k} \log{(x+1)} $, $x \log x$, where $k\in \mathbb{N}$.   
\end{remark}
\begin{example}
A sequence $(a_n)_{n \geq 1}$ of positive real numbers is called \textbf{lacunary} if there exists a constant $D > 1$ such that $\frac{a_{n+1}}{a_n} \geq D$ for all $n \in \mathbb{N}$.

A lacunary sequence satisfies all the conditions of Theorem~\ref{thm:min_spacing}.
\begin{enumerate}
    \item It consists of distinct positive real numbers (since $a_{n+1} \geq D a_n > a_n$ for all $n$).
    \item For $m\neq n \in \mathbb{N}$, using lacunarity, we obtain
\begin{align}
|a_m - a_n| \geq a_{\max(m,n)} - a_{\max(m,n)-1} \geq (D - 1)a_{\max(m,n)-1} \geq (D - 1)a_1.
\end{align}
That is, $(a_n)_{n \geq 1}$ satisfies the condition $\displaystyle\inf_{m \neq n} |a_m - a_n| \geq c$ with $c = (D - 1)a_1>0$.
\end{enumerate}

Therefore, Theorem~\ref{thm:min_spacing} applies to lacunary sequences.
Since the sequence $a_n\coloneqq c^{n}$, where $c>1$ is a fixed constant, is lacunary,  Theorem~\ref{thm:min_spacing} also applies to such sequences, in particular, to the sequences $\left(e^{n}\right)_{n \geq 1}$, $\left(\pi^{n}\right)_{n \geq 1}$, and $\left((\frac{3}{2})^{n}\right)_{n \geq 1}$. 
\end{example}

\begin{remark}
 Let $f$ satisfy the hypotheses in Corollary~\ref{part 0}, and let $(a_n)_{n \geq 1}$ be a lacunary sequence. Then, for almost all $\alpha \in [0,1]$ and for all sufficiently large $N$, we have 
 $$\delta_{\mathrm{min}}^{\alpha}((f(a_n))_{n\geq 1},N)\geq \frac{1}{\widetilde{D}_N \log N (\log_{2}N)^{1+\epsilon}}, \text{ and}$$
 $$\delta_{\mathrm{min}}^{\alpha}((f(a_n))_{n\geq 1},N)\leq 1-\frac{1}{\widetilde{D}_N \log N (\log_{2}N)^{1+\epsilon}},$$
 where $\widetilde{D}_N=\#(f(A_N)-f(A_N))^{+}$.
\end{remark}

\begin{cor}\label{part tending to 0}
Let \((a_n)_{n \geq 1}\) be a sequence of distinct positive real numbers such that $\displaystyle\inf_{\substack{1\leq  m \neq n \leq N}} |a_m - a_n| \geq c(N)$, where $(c(N))_{N\geq1}$ is a monotone decreasing sequence of positive reals. Let $f$ be a differentiable function on $[\displaystyle\inf_{n}a_{n},\infty)$, which is positive-valued on $(a_n)_{n \geq 1}$ and for which there exists $\delta_{0} >0$ such that $|f'(x)|\geq \delta_0$ for all $x \in [\displaystyle\inf_{n}a_{n},\infty)$.
 Then, for almost all $\alpha \in [0,1]$ and for all sufficiently large $N$, we have 
    
$\textbf{(a)}$ $$\delta_{\mathrm{min}}^{\alpha}((f(a_n))_{n\geq 1},N)\geq \frac{c(N)\delta_{0}}{\widetilde{D}_N \log N (\log_{2}N)^{1+\epsilon}}, \text{ and}$$

$\textbf{(b)}$ $$\delta_{\mathrm{min}}^{\alpha}((f(a_n))_{n\geq 1},N)\leq 1-\frac{c(N) \delta_{0}}{\widetilde{D}_N \log N (\log_{2}N)^{1+\epsilon}}, $$
where \(  \widetilde{D}_N = \#(f(A_N) - f(A_N))^+ \).
\end{cor}
\begin{proof}
    By the mean value theorem, for any $m \neq n$, there exists $\xi \in [\inf (a_m,a_n),\sup (a_m,a_n)]$ such that  $$|f(a_m)-f(a_n)|=|f'(\xi)||a_m -a_n|\geq \delta_{0} |a_m -a_n|.$$

     Therefore, for all $1\leq m', n' \leq N$ with $m'\neq n'$, 
    $$|f(a_{m'})-f(a_{n'})|\geq \delta_{0} |a_{m'} -a_{n'}|\geq\delta_{0} \inf_{1 \leq  m\neq n \leq N} |a_{m} -a_{n}|\geq c (N) \delta_{0}.$$ 
 In particular, $\displaystyle\inf_{\substack{1\leq  m \neq n \leq N}} |f(a_m) -f( a_n)| \geq d(N) $, where $d(N)=c(N) \delta_{0}$.

Since $c(N)$ is a monotone decreasing sequence of positive reals, so is $d(N)$. Since $|f'(x)| \geq \delta_0 > 0$ for all $x$ and $f$ is differentiable, $f$ is injective. Thus, $(f(a_{n}))_{n \geq 1}$ is a sequence of distinct positive real numbers satisfying $\displaystyle\inf_{1 \leq m\neq n \leq N}|f(a_m)-f(a_n)|\geq d(N)$. We obtain the required bound in part $\textbf{(a)}$ by applying Theorem~\ref{thm:main}.
The proof of part $\textbf{(b)}$ follows from Theorem~\ref{thm:main}.
\end{proof}
\begin{example}
Let $\theta <1$ and $f(x)=x^{\theta}$. Then $f$ is differentiable and positive-valued on $\mathbb{N}$, satisfying $|f'(x)|=|\theta| |x|^{\theta -1} \geq |\theta| N^{\theta -1} $ for all $x \in [1,N]$. Let $a_n =f(n)=n^{\theta}$ for all $n \in \mathbb{N}$. Therefore, by the mean value theorem, $\displaystyle\inf_{\substack{1\leq  m \neq n \leq N}} |a_m - a_n| \geq \displaystyle\inf_{\substack{1\leq  m \neq n \leq N}} |\theta| N^{\theta -1}|m-n| \geq |\theta| N^{\theta -1} $. The assumption $\theta <1$ ensures that the sequence $(c(N))_{N \geq 1}$, where $c(N):=|\theta| N^{\theta -1}$ is a monotone decreasing.

Therefore, using Theorem~\ref{thm:main}, we obtain that for almost all $\alpha \in [0,1]$ and for all sufficiently large $N$,  
    
$\textbf{(a)}$ $$\delta_{\mathrm{min}}^{\alpha}((n^\theta)_{n\geq 1},N)\geq \frac{|\theta| N^{\theta -1}}{D_N \log N (\log_{2}N)^{1+\epsilon}}\geq \frac{|\theta|}{N^{3-\theta} \log N (\log_{2}N)^{1+\epsilon}}, \text{ and}$$

$\textbf{(b)}$ $$\delta_{\mathrm{min}}^{\alpha}((n^\theta)_{n\geq 1},N)\leq 1-\frac{|\theta| N^{\theta -1}}{D_N \log N (\log_{2}N)^{1+\epsilon}}\leq 1-\frac{|\theta| }{N^{3-\theta} \log N (\log_{2}N)^{1+\epsilon}}, \text{ where}$$
 \( A_N = \{1^{\theta},2^{\theta}, \dots, N^{\theta}\} \) and \( D_N = \#(A_N - A_N)^+ \).

In particular, when $\theta=\frac{1}{2}$, we have 
 $$\delta_{\mathrm{min}}^{\alpha}((\sqrt{n})_{n\geq 1},N)\geq \frac{1}{2N^{5/2} \log N (\log_{2}N)^{1+\epsilon}}$$
for almost all $\alpha \in [0,1]$ and for all sufficiently large $N$. We remark that the lower bound we obtained for the minimal gap of the sequence $(\sqrt{n})_{n \geq 1}$ is weaker than the optimal bound, which is $\sim \frac{1}{2N^{3/2}}$ $($see \cite[Theorem $1.9$]{Regavim}$)$. For any $\theta $ which is very close to $1$ $(\theta <1)$, the lower bound we obtain for the minimal gap of the sequence $(n^{\theta})_{n \geq 1}$ is nearly optimal, which is believed to be $\sim \frac{1}{N^2}$.
\end{example}

\begin{cor}\label{part tending to 0 part 2}
Let \((a_n)_{n \geq 1}\) be a sequence of distinct positive real numbers such that
\begin{align}
\inf_{\substack{1\leq  m \neq n \leq N}} |a_m - a_n| \geq N^{\delta-1}
\end{align}
for some $\delta \in [0,1)$. Let $f$ be a differentiable function on $[\displaystyle\inf_{n}a_{n},\infty)$, which is positive-valued on $(a_n)_{n \geq 1}$ and for which there exists $\delta_{0} >0$ such that $|f'(x)|\geq \delta_0$ for all $x \in [\displaystyle\inf_{n}a_{n},\infty)$.
 Then, for almost all $\alpha \in [0,1]$ and for all sufficiently large $N$, we have 
    
$\textbf{(a)}$ $$\delta_{\mathrm{min}}^{\alpha}((f(a_n))_{n\geq 1},N)\geq \frac{N^{\delta -1}\delta_{0}}{\widetilde{D}_N \log N (\log_{2}N)^{1+\epsilon}} \geq \frac{\delta_{0}}{N^{3-\delta} \log N (\log_{2}N)^{1+\epsilon}}, \text{ and}$$

$\textbf{(b)}$ $$\delta_{\mathrm{min}}^{\alpha}((f(a_n))_{n\geq 1},N)\leq 1-\frac{N^{\delta -1} \delta_{0}}{\widetilde{D}_N \log N (\log_{2}N)^{1+\epsilon}} \leq 1-\frac{\delta_{0}}{N^{3-\delta} \log N (\log_{2}N)^{1+\epsilon}}.$$
where \(  \widetilde{D}_N = \#(f(A_N) - f(A_N))^+ \).
\end{cor}
\begin{proof}
The proof follows immediately from Corollary~\ref{part tending to 0} taking $c(N) = N^{\delta - 1}$, where $\delta < 1$.
\end{proof}
\begin{example}
    Let $f(x)=\log (1+x)$ and $a_n=f(n)$. Then, the mean value theorem ensures that there exists $c_{m,n} \in (\inf (m,n), \sup (m,n))$ such that $|a_m -a_n|=|f'(c_{m,n})||m-n|\geq \frac{1}{1+c_{m,n}}$, that is, $\displaystyle\inf_{1 \leq m \neq n \leq N} |a_m -a_n| \geq \frac{1}{N+1}\geq \frac{1}{2N}$. Therefore, for any \( \epsilon > 0 \) and for almost all \( \alpha \in [0,1] \), there exists \( N_0(\alpha) \) such that for all \( N \geq N_0(\alpha) \),
     \(\delta^{\alpha}_{\min}((\log(1+n))_{n \geq 1},N) \geq \dfrac{1}{2D_N N\log N (\log_2 N)^{1+\epsilon}}\geq \dfrac{1}{2N^{3} \log N (\log_2 N)^{1+\epsilon}}\), where \( A_N = \{\log 2,\log 3, \dots, \log (1+N)\} \) and \( D_N = \#(A_N - A_N)^+ \).
\end{example}

\section{Proofs of Theorems~\ref{upper bound},~\ref{upper bound independent of a_N} and~\ref{upper bound in terms of H_N for sufficiently large N}: Upper Bounds}  \label{Proof of Theorem 6.4 upper bound}
The convergence part of the Borel–Cantelli lemma holds unconditionally, whereas the divergence part requires the hypothesis of pairwise independence. Consequently, proving divergence in the context of metric number theory often involves greater technical difficulty than proving convergence (see \cite{BV2023} for a discussion). In Section~\ref{outline}, we briefly outline the heuristics behind the proof. In Section~\ref{upper bound depending on size a_N}, we provide a proof of Theorem~\ref{upper bound}, where the upper bound depends on the size of $a_N$. Section~\ref{Overlap estimates and auxiliary lemmas} compiles the necessary technical lemmas and overlap estimates required for the subsequent analysis in Theorem~\ref{upper bound independent of a_N}, and in Section~\ref{section upper bound independent of aN}, we finally present the proof of Theorem~\ref{upper bound independent of a_N}, where the upper bound is independent of the size of $a_N$. In Section~\ref{Subsection:upper bound in terms of H_N for sufficiently large N}, we prove Theorem~\ref{upper bound in terms of H_N for sufficiently large N}, where the upper bound result holds for all except finitely many $N$.

\subsection{Outline and heuristics}\label{outline} 
The proofs follow the modern approach to metric Diophantine approximation problems. The fundamental challenge is that the approximation sets under consideration lack the stochastic independence required for a direct application of the second Borel-Cantelli lemma, as demonstrated by the classical Duffin-Schaeffer counterexample (see \cite{Duffin1941KhintchinesPI}).

Instead, we rely on the concept of ``quasi-independence on average'' (See Lemma~\ref{chung erdos inequality} and \cite{BDV2006, BV2010, BV2023}), where the idea is to define the $S_n$ such that the sum of their measures diverges, while simultaneously controlling the measure of their pairwise overlaps ($S_m \cap S_n$ for $m \neq n$) to apply a suitable convergence criterion. Following Rudnick's strategy, we employ $L^2$ methods that naturally lead to the study of greatest common divisor (GCD) sums, which play a central role in this area.

Recent optimal bounds for GCD sums obtained by de la Bretèche and Tenenbaum (see \cite{delabretèche}) introduce a specific growth factor that determines the threshold for successful application of this method. Specifically, for the $L^2$ method to succeed, the sum of measures of our approximation sets must exceed this GCD sum growth factor. This explains the precise form of the conditions appearing in our theorems.
\subsection{Proof of Theorem~\ref{upper bound}}\label{upper bound depending on size a_N}
We begin by stating a conjecture that is now a theorem due to the work of Koukoulopoulos and Maynard (see \cite{KoukouMaynard}). This result will be useful in the proof of Theorem~\ref{upper bound}.
\begin{thm}[Duffin-Schaeffer Conjecture]\label{Duffin-Schaeffer Conjecture}
        Let $\psi: \mathbb{N} \to \mathbb{R}_{\geq 0}$ be a function such that $\displaystyle\sum_{q=1}^{\infty} \frac{\psi(q)\phi(q)}{q}=\infty$. Let $\mathcal{A}$ be the set of $\alpha \in [0,1]$ for which the inequality 
        \begin{equation}\label{DS}
            \left|  \alpha -\frac{a}{q}\right|\leq \frac{\psi(q)}{q},\, (a,q)=1, \,q\geq 1
        \end{equation} has infinitely many solutions in $a$ and $q$. Then $\mathcal{A}$ has the Lebesgue measure $1$.
   \end{thm}
\begin{remark}
    Here we restrict our attention to $\alpha \in [0,1]$ only, because if $\alpha$ has an approximation given by inequality \eqref{DS}, then so does $\alpha+r$ for any integer $r$, with $a$ replaced by $a+rq$. Note that $(a,q)=1$ implies $(a+rq,q)=1$.
\end{remark}
\begin{lemma}\label{H_N bound for general sequence}
Let $(a_n)_{n \geq 1}$ be a sequence of real numbers such that $\displaystyle\inf_{\substack{1 \leq m \neq n \leq N}} |a_m - a_n| \geq 1$ for all $N \in \mathbb{N}$.
Define the finite set $A_N = \{a_1, \dots, a_N\}$ and let $H_N = \#\lfloor (A_N - A_N)^+ \rfloor$. Then $H_N $ satisfies the recurrence
$\lceil \frac{N}{2} \rceil  \leq H_{N+1} \leq H_N + N$. Consequently, $H_{N+1} \leq 3 H_{N}+1\leq 4 H_{N}$ holds for $N \geq 3$.
\end{lemma}

\begin{proof}
We begin by expressing the positive difference set for $A_{N+1}$ as 
\begin{align}
(A_{N+1} - A_{N+1})^+ = (A_N - A_N)^+ \cup \{ |a_{N+1} - a| : a \in A_N \}.
\end{align}
Taking the floor of both sides (applied element-wise to the set) gives
\begin{align}\label{counting AN first}
\lfloor (A_{N+1} - A_{N+1})^+ \rfloor = \lfloor (A_N - A_N)^+ \rfloor \cup S,
\end{align}
where $S = \{ \lfloor |a_{N+1} - a| \rfloor : a \in A_N \}$. It follows that
\begin{align}
H_{N+1} = \#\lfloor (A_{N+1} - A_{N+1})^+ \rfloor \leq \#\lfloor (A_N - A_N)^+ \rfloor + \#S \leq  H_N + N,
\end{align}
which establishes the upper bound.

For the lower bound, consider the set of floor values of distances from \( a_1 \):
\begin{align}
B = \{ \lfloor |a_j - a_1| \rfloor : 2 \leq j \leq N \}.
\end{align}
Each element in \( B \) is an integer and we claim that for each integer \( n \in \mathbb{N} \), there can be at most two indices \( j \geq 2 \) such that \( \lfloor |a_j - a_1| \rfloor = n \).

Suppose for contradiction that there exist three distinct indices \( j, k, \ell \geq 2 \) with
\begin{align}
    \lfloor |a_j - a_1| \rfloor = \lfloor |a_k - a_1| \rfloor = \lfloor |a_\ell - a_1| \rfloor = n.
\end{align}
This implies \( n \leq |a_j - a_1|, |a_k - a_1|, |a_\ell - a_1| < n+1 \).
By the pigeonhole principle, at least two of these points must be on the same side of \( a_1 \) (that is, both larger or both smaller than \( a_1 \)). Without loss of generality, assume that both \( a_j \) and \( a_k \) are larger than \( a_1 \) (the case where both are smaller is similar). Then \( a_j, a_k \in [a_1 + n, a_1 + n + 1) \).

But then \( |a_j - a_k| < 1 \), contradicting the assumption that \( \displaystyle\inf_{m \neq n} |a_m - a_n| \geq 1 \). Hence, for each \( n \in \mathbb{N} \), there are at most two points \( a_j \) (with \( j \geq 2 \)) such that \( \lfloor |a_j - a_1| \rfloor = n \).

Therefore, the number of distinct values in \( B \) is at least \( \left\lceil \frac{N-1}{2} \right\rceil \). Since \( B \subseteq \lfloor |A_N - A_N| \rfloor \), it follows that the number of distinct distances \( H_N \) satisfies $H_N \geq \left\lceil \frac{N-1}{2} \right\rceil$.
This proves the first part of the lemma. 

For $N\geq 3$, the lower bound gives $H_{N} \geq \frac{N-1}{2} (\geq 1)$. Substituting this into the upper bound gives
    \begin{align}
        H_{N+1} \leq H_{N}+2\left(\frac{N-1}{2}\right)+1\leq H_{N}+2H_{N}+1\leq  4H_{N}.
    \end{align}
    This completes the proof.
\end{proof}

\begin{remark}
Lemma~\ref{H_N bound for general sequence} shows that $H_{N+1} \leq H_N + N$, and by iterating this recurrence, we obtain the trivial bound $H_N \leq \displaystyle\sum_{k=1}^{N-1} k = \binom{N}{2}$.  
\end{remark}
We now establish a variant of Lemma~\ref{H_N bound for general sequence} for strictly increasing sequences, which yields an improved bound for $H_N$. Although this lemma is not necessary for the proof of Theorem~\ref{upper bound}, it will be used in the proof of one of its corollaries.

\begin{lemma}\label{$H_N$ bound for increasing sequence}
Let $(a_n)_{n \geq 1}$ be a strictly increasing sequence of real numbers such that $\displaystyle\inf_{\substack{1 \leq m \neq n \leq N}} |a_m - a_n| \geq 1$ for all $N \in \mathbb{N}$. Then $H_N $ satisfies the recurrence
$N \leq H_{N+1} \leq H_N + N$. Consequently, $H_{N+1} \leq 2H_{N}+1\leq 3 H_{N}$ holds for $N \geq 2$.
\end{lemma}

\begin{proof}
We express the floor of positive difference set for $A_{N+1}$ as 
\begin{align}\label{counting AN}
\lfloor (A_{N+1} - A_{N+1})^+ \rfloor = \lfloor (A_N - A_N)^+ \rfloor \cup S,
\end{align}
where $S = \{ \lfloor a_{N+1} - a \rfloor : a \in A_N \}$. We now claim that the set $S$ contains $N$ \emph{distinct} elements.

Suppose, for contradiction, that this is not true. Then, there exist indices $1 \leq i \neq j \leq N$ such that
\begin{align}
    \lfloor a_{N+1} - a_i \rfloor = \lfloor a_{N+1} - a_j \rfloor = k \quad \text{(say)}.
\end{align}
By the definition of the floor function, we can write $a_{N+1}-a_{i}=k+t_{i}$ and $a_{N+1}-a_{j}=k+t_{j}$, where $0\leq t_{i}, t_{j} <1$.
Therefore,
\begin{align}
    |a_j - a_i |= |(a_{N+1} - a_i) - (a_{N+1} - a_j)| =| (k + t_i) - (k + t_j) |= |t_i - t_j|\in [0, 1).
\end{align}
This implies $\lfloor |a_i - a_j| \rfloor = 0$ for $i \neq j$, which contradicts the initial assumption that $\inf_{m \neq n} |a_m - a_n| \geq 1$. This contradiction proves that $S$ has $N$ distinct elements. Moreover, since the sequence is strictly increasing, all elements of $S$ are positive, and we have $H_{N+1} \geq \#S = N$.

The upper bound follows directly from Lemma~\ref{H_N bound for general sequence}. This proves the first part of the lemma. 

For $N\geq 2$, the lower bound gives $H_{N} \geq N-1 (\geq 1)$. Substituting this into the upper bound gives
    \begin{align}
        H_{N+1} \leq H_{N}+(N-1)+1\leq 2H_{N}+1\leq  3H_{N}.
    \end{align}
    This completes the proof.
\end{proof}
\begin{remark}
    Alternatively, to show $H_N \geq  N-1$ for an increasing sequence $(a_n)_{n \geq 1}$ with $\displaystyle\inf_{m \neq n} |a_m - a_n| \geq 1$, consider the set
  $\mathcal{A} \coloneqq \{ \lfloor a_j - a_1 \rfloor : 2 \leq j \leq N \}$. 
    Since $a_j - a_1 = (a_j - a_{j-1}) + (a_{j-1} - a_1) \geq 1 + (a_{j-1} - a_1)$, we have $\lfloor a_j - a_1 \rfloor \geq \lfloor a_{j-1} - a_1 \rfloor + 1$.
    Therefore, the values $\lfloor a_j - a_1 \rfloor$ increase strictly with $j$, and therefore are all distinct. Since $\mathcal{A} \subseteq \lfloor (A_N - A_N)^+ \rfloor$, it follows that $H_N \geq  N-1$.
\end{remark}
\textbf{Proof of Theorem~\ref{upper bound}.}    We arrange the elements of $A=\displaystyle\cup_{N\geq 1}\lfloor  (A_N -A_N)^{+}\rfloor$ into a sequence. Thus, let $(Z_n)_{n\geq 1}$ be a sequence of distinct positive integers such that we have
    \begin{align}\label{z_n equation}
        \{Z_n: 1\leq n\leq H_N\}=\lfloor(A_N -A_N)^{+}\rfloor \text{ for all }N.
    \end{align}
    The sequence $(Z_n)_{n \geq 1}$ is not uniquely defined, as $H_{N+1}-H_{N}$ can be strictly larger than $1$. In fact, it can be as large as $N$, but the proof works irrespective of the version of $(Z_n)_{n \geq 1}$ we take as long as \eqref{z_n equation} is satisfied.
To reduce the size of overlaps $S_{m} \cap S_{n}$, we use the same strategy as is used in the coprime setup of the Duffin-Schaeffer conjecture. We replace $S_{m}$ and $S_{n}$ with modified sets that preserve most of the measure of the original sets, but remove those parts that are excessively responsible for the overlaps. To incorporate this idea into our situation, we consider the reduced sets
\begin{align}\label{Sn coprime}
S_{n}^{\text {coprime }}=[0,1] \cap\left(\bigcup_{\substack{0 \leq a \leq Z_n\\ \gcd\left(a, Z_n\right)=1}}\left(\frac{a}{Z_{n}}-\frac{\psi(n)}{Z_{n}}, \frac{a}{Z_{n}}+\frac{\psi(n)}{Z_{n}}\right)\right),
\end{align}
where $\displaystyle\psi(n)= \frac{\log _{2} Z_{n}}{n \log n \log _{2} n}$ for all $n$. This reflects the fact that the overlaps $S_{m} \cap S_{n}$ come mainly from intervals in $S_{m}$ and $S_{n}$, which are centred around points $a / Z_{m}$ and $b / Z_{n}$, respectively, for which $a / Z_{m}=$ $b / Z_{n}$, so that either $\gcd\left(a, Z_{m}\right)>1$ or $\gcd\left(b, Z_{n}\right)>1$. 

\noindent
Let 
\begin{align}\label{Sn new}
S_{n}=[0,1] \cap\left(\bigcup_{\substack{0 \leq a \leq Z_n}}\left(\frac{a}{Z_{n}}-\frac{\psi(n)}{Z_{n}}, \frac{a}{Z_{n}}+\frac{\psi(n)}{Z_{n}}\right)\right).
\end{align}
Then
$$\lambda(S_n)=\frac{2\psi(n)}{Z_n}Z_n= \frac{2\log _{2} Z_{n}}{n \log n \log _{2} n}.$$

It is known that ${\varphi(T) / T} \gg {(\log_2 T)^{-1}}$, that is, there exist a positive constant $C_0$ and an $N_0^{(1)} \in \mathbb{N}$ such that ${\varphi(Z_n) / Z_n} \geq {C_0(\log_2 Z_n)^{-1}}$ for all $n \geq N_0^{(1)}$. Observe that this is a function of $Z_{n}$, which means that it depends on the actual size of the elements of the floored positive difference set $\lfloor (A_{N}-A_{N})^{+}\rfloor$, and not a function of the cardinality of the difference set, that is, $n$.

We first assume that $\psi(n) \geq \frac{1}{2}$ for infinitely many $n$. Since $\| \alpha z_n\| \leq \frac{1}{2}$, we have $\| \alpha z_n\| \leq \psi(n)$ for infinitely many $n$. Otherwise, $\psi(n) \geq \frac{1}{2}$ for only finitely many $n$, that is, there exists $N_0^{(2)}$ such that $\psi(n) < \frac{1}{2}$ for all $n \geq N_0^{(2)}$, which ensures that $2\psi(n) \varphi\left(Z_n\right) / Z_n < 1$.
Let $N_0 \coloneqq \max (N_{0}^{(1)},N_{0}^{(2)})$. Therefore,
\begin{align}
  \sum_{n=N_0}^{\infty}\lambda\left(S_{n}^{\text {coprime }}\right)=& \sum_{n=N_0}^{\infty} \min (2\psi(n) \varphi\left(Z_n\right) / Z_n,1)\\
  = &\sum_{n=N_0}^{\infty}2\psi(n) \varphi\left(Z_n\right) / Z_n\\
  \geq & \sum_{n=N_0}^{\infty}\frac{2C_0\psi(n)}{ \log _{2} Z_n }=\sum_{n=N_0}^{\infty}\frac{2C_0}{n \log n \log _{2} n}=+\infty.
\end{align} 
Let $\psi^{*}: \mathbb{N} \to \mathbb{R}_{\geq 0}$ be defined by
$$\psi^{*}(q)=\begin{cases}
    \psi(n)  &\quad\text{if }q=Z_n \text{ for some }n  \in \mathbb{N}, \\        0 &\quad\text{otherwise.} 
\end{cases}$$
Since $(Z_n)_{n \geq 1}$ is a sequence of distinct positive integers, the map $\psi^{*}$ is well-defined. 
Therefore,
\begin{align}
  \sum_{q=1}^{\infty} \frac{\psi^{*}(q) \phi(q)}{q}=\sum_{n=1}^{\infty}\frac{\psi(n) \phi(Z_n)}{Z_n} = \frac{1}{2}\sum_{n=N_0}^{\infty}\lambda\left(S_{n}^{\text {coprime }}\right)=\infty.
\end{align}

We note that $\psi^{*}: \mathbb{N} \to \mathbb{R}_{\geq 0}$ is a function such that $\displaystyle\sum_{q=1}^{\infty} \frac{\psi^{*}(q)\phi(q)}{q}=\infty$. Therefore, using the Duffin-Schaeffer conjecture (see Theorem~\ref{Duffin-Schaeffer Conjecture}), there exists a set $\mathcal{K'}\subseteq [0,1]
$ with $\lambda(\mathcal{K'})=1$ (without loss of generality, we can assume that $\mathcal{K'}\cap \mathbb{Q}=\phi$) such that for all $\alpha \in \mathcal{K'}$, the inequality
\begin{align}\label{Ds1}
    \left|  \alpha -\frac{a}{q}\right|\leq \frac{\psi^{*}(q)}{q}
\end{align}
has infinitely many coprime solutions $a$ and $q$.

We note that if $q\neq Z_n$ for some $n \in \mathbb{N}$, then $\psi^{*}(q)=0$, and \eqref{Ds1} gives $|\alpha -\frac{a}{q}|=0$, that is, $\alpha =\frac{a}{q}\in \mathbb{Q}$, which is a contradiction.
 In other words, $q=Z_n$ for some $n \in \mathbb{N}$.

Therefore, for all $\alpha \in \mathcal{K}$, where $\mathcal{K}\subseteq [0,1]
$ and $\lambda(\mathcal{K})=1$, the inequality
\begin{align}
    &\left|  \alpha -\frac{a}{Z_n}\right|\leq \frac{\psi^{*}(Z_n)}{Z_n}\\
    \implies &\left|  \alpha -\frac{a}{Z_n}\right|\leq \frac{\psi(n)}{Z_n}\\
    \implies &\left|  \alpha Z_n -a\right|\leq \psi(n)\\
    \implies &\left\|  \alpha Z_n\right\|\leq \psi(n)
    \end{align}
    has infinitely many coprime solutions $a$ and $Z_n$, that is, $\left\|  \alpha Z_n\right\|\leq \frac{\log _{2} Z_{n}}{n \log n \log _{2} n}$ for infinitely many $n$. This shows that there exists a sequence $((t(n))_{n \geq 1}$ of distinct natural numbers such that $\left\|  \alpha Z_{t(n)}\right\|\leq \frac{\log _{2} Z_{t(n)}}{t(n) \log t(n) \log _{2} t(n)}$ for all $n \in \mathbb{N}$.

By the construction of $Z_n$, for every $n$, there is a unique $N$ such that $H_{N}< t(n) \leq H_{N+1}$, and Lemma~\ref{H_N bound for general sequence} shows that $t(n) > H_{N} \geq \frac{1}{4}H_{N+1}\geq \frac{N}{8}$. 

Since $B_{N+1}= \displaystyle \max_{1\leq m\neq n\leq  N+1} \lfloor |a_{m}-a_{n}|\rfloor= \max \lfloor (A_{N+1}-A_{N+1})^{+} \rfloor=\max \{Z_n: 1\leq n\leq H_{N+1}\}$, we have $Z_{t(n)} \leq B_{N+1}$. Therefore, for all $n \in \mathbb{N}$ with $t(n) \in (H_{N},H_{N+1}]$, we have
\begin{align}
 \left\|  \alpha Z_{t(n)}\right\|\leq \frac{\log _{2} Z_{t(n)}}{t(n) \log t(n) \log _{2} t(n)} \leq \frac{\log _{2} B_{N+1}}{H_{N} \log (\frac{N}{8}) \log _{2} (\frac{N}{8})}.
\end{align} 
In particular, there are infinitely many $N$ such that
\begin{align}
   \lfloor \delta^{\alpha}_{\min }\rfloor(N+1)=&\displaystyle \min_{1\leq m\neq n\leq  N+1}\| \alpha\lfloor |a_{m}-a_{n}| \rfloor\| 
   = \displaystyle \min_{l\leq  H_{N+1}}\|\alpha Z_{l}\|
   \leq \|\alpha Z_{t(n)}\| \leq \frac{\log _{2} B_{N+1}}{H_{N} \log (\frac{N}{8}) \log _{2} (\frac{N}{8})}.
\end{align}
Using Lemma~\ref{H_N bound for general sequence}, we see that there are infinitely many $N$ such that
\begin{align}\label{mimick}
   \lfloor \delta^{\alpha}_{\min }\rfloor(N)\leq \frac{\log _{2} B_{N}}{H_{N-1} \log (  \frac{N-1}{8}) \log _{2}(  \frac{N-1}{8})}\leq \frac{4\log _{2} B_{N}}{H_{N} \log (\frac{N}{9}) \log _{2} (\frac{N}{9})}.
\end{align}
We now prove that $B_N \leq \lfloor \max(a_1,\ldots,a_N) \rfloor$.

For any $1 \leq m \neq n \leq N$, we have
$|a_m - a_n| \leq \max(a_1,\ldots,a_N) - \min(a_1,\ldots,a_N)$.
Taking floors, we get
\begin{align}
\lfloor |a_m - a_n| \rfloor \leq \lfloor \max(a_1,\ldots,a_N) - \min(a_1,\ldots,a_N) \rfloor.
\end{align}
Since we have all $a_i > 0$, we have $\min(a_1,\ldots,a_N) > 0$, so $\lfloor |a_m - a_n| \rfloor \leq \lfloor \max(a_1,\ldots,a_N) \rfloor$, and we obtain the upper bound claimed for $B_N$.

\begin{cor}\label{upper bound increasing sequence}
    Let $(a_n)_{n \geq 1}$ be an increasing sequence of distinct positive real numbers such that
$\displaystyle\inf_{\substack{1\leq  m \neq n \leq N}} |a_m - a_n| \geq 1$ for all $N \in \mathbb{N}$. Then, for almost all $\alpha \in [0,1]$, we have 
    \begin{align}
            \lfloor \delta^{\alpha}_{\min }\rfloor(N)  \leq  \frac{3\log _{2} B_{N}}{H_{N} \log (  \frac{N}{4}) \log _{2} (  \frac{N}{4})} \text{ for infinitely many $N$,}
    \end{align}
    where $ \lfloor \delta^{\alpha}_{\min }\rfloor(N) =\displaystyle \min_{1\leq m\neq n\leq  N}\| \alpha\lfloor |a_{m}-a_{n}| \rfloor\| $,  $B_N= \displaystyle \max_{1\leq m\neq n\leq  N} \lfloor |a_{m}-a_{n}|\rfloor$, and $H_N =  \#\lfloor(A_N -A_N)^{+}\rfloor$. 
    Since $B_N \leq \lfloor a_N\rfloor$, the above statement also holds when $B_N$ is replaced by $\lfloor a_N\rfloor$.
\end{cor}
    \begin{proof}
 Lemma~\ref{$H_N$ bound for increasing sequence} shows that $ H_{N} \geq \frac{1}{3}H_{N+1}\geq \frac{N}{3}$. Then, by a calculation similar to \eqref{mimick}, we obtain
\begin{align}
   \lfloor \delta^{\alpha}_{\min }\rfloor(N)\leq \frac{\log _{2} B_{N}}{H_{N-1} \log (  \frac{N-1}{3}) \log _{2}(  \frac{N-1}{3})}\leq \frac{3\log _{2} B_{N}}{H_{N} \log (\frac{N}{4}) \log _{2} (\frac{N}{4})}
\end{align}
for almost all $\alpha \in [0,1]$ and for infinitely many $N$.
\end{proof}

\begin{cor}\label{upper bound of AMB}
    Let $(a_n)_{n \geq 1}$ be a sequence of distinct positive integers. Then, for almost all $\alpha \in [0,1]$, we have 
    \begin{align}
    \delta^{\alpha}_{\min }(N)  \leq  \frac{4\log _{2}  \overline{B}_N}{D_{N} \log (\frac{N}{9}) \log _{2} (\frac{N}{9})} \text{ for infinitely many $N$,}
    \end{align}
    where $\overline{B}_N:=  \max(a_1,\ldots, a_N)$, and $D_N=  \#(A_N -A_N)^{+}$.
\end{cor}
\begin{proof}
    Since for a sequence $(a_n)_{n \geq 1}$ of distinct positive integers we have $\displaystyle\inf_{1\leq m \neq n\leq N}|a_{m}-a_{n}|\geq 1$, and $ \lfloor \delta^{\alpha}_{\min }\rfloor(N)=\delta_{\min }^{\alpha}(N)$, we recover \cite[Theorem 1 (Equation 5)]{ABM} from Theorem~\ref{upper bound}.
\end{proof}

\begin{remark}
    We observe that although in \cite[Theorem 1 (Equation 5)]{ABM}, the result is stated as 
    \begin{align}
            \delta^{\alpha}_{\min }(N)  \leq  \frac{\log _{2}  a_N}{D_{N} \log  N \log _{2}  N} \text{ for infinitely many }N,
    \end{align}
   we believe that for any sequence of distinct positive integers $(a_n)_{n \geq 1}$, the correct bound should be
      \begin{align}
            \delta^{\alpha}_{\min }(N)  \leq  \frac{4\log _{2}  \max(a_1,\ldots, a_N)}{D_{N} \log (\frac{N}{9}) \log _{2} (\frac{N}{9})} \text{ for infinitely many }N.
    \end{align}
 Similarly, for any increasing sequence of distinct positive integers $(a_n)_{n \geq 1}$, the correct bound should be
    \begin{align}
     \delta^{\alpha}_{\min }(N)  \leq  \frac{3\log _{2}  a_N}{D_{N} \log (  \frac{N}{4}) \log _{2} (  \frac{N}{4})} \text{ for infinitely many }N.
    \end{align}   
These are derived from Corollaries~\ref{upper bound of AMB} and~\ref{upper bound increasing sequence}, respectively.
\end{remark}
\begin{remark}\label{H_N=T_N}
      Let $T_N = \#(\lfloor A_N - A_N \rfloor)^+$ and $H_N =  \#\lfloor(A_N -A_N)^{+}\rfloor$ be defined as in Section~\ref{Notation}. Since $|a_m - a_n| \geq 1$ for all positive integers $m, n$ with $m \neq n$, both sets $(\lfloor A_N - A_N \rfloor)^+$ and $\lfloor (A_N - A_N)^+ \rfloor$ are equal to $\{ \lfloor a_m - a_n \rfloor : a_m > a_n \}$, and we obtain $T_N = H_N$. 
\end{remark}
\begin{prop}\label{twisted upper bound}
    Let $(a_n)_{n \geq 1}$ be a sequence of distinct positive real numbers such that for all $N \in \mathbb{N}$. Then, for almost all $\alpha \in [0,1]$, we have 
    \begin{align}
         \lfloor \widetilde{\delta}_{\min}^{\alpha} \rfloor(N)            
        \leq   \frac{4\log _{2} \widetilde{B}_{N}}{H_{N} \log (\frac{N}{9}) \log _{2} (\frac{N}{9})} \text{ for infinitely many }N,
    \end{align}
    where $\lfloor \widetilde{\delta}_{\min}^{\alpha} \rfloor(N)=\displaystyle\min_{1 \leq m \neq n \leq N} \left\| \alpha \left\lfloor a_m - a_n \right\rfloor \right\|$, $\widetilde{B}_{N} = \displaystyle \max_{1\leq m\neq n\leq  N} |\left\lfloor a_{m}-a_{n}\right\rfloor|$, and $H_N =  \#\lfloor(A_N -A_N)^{+}\rfloor$.
    Since $\widetilde{B}_N \leq \max(\lfloor  a_1 \rfloor,\ldots, \lfloor  a_N \rfloor )+1=\lfloor \max(a_1,\ldots, a_N)\rfloor+1$, the above statement also holds when $\widetilde{B}_N$ is replaced by $\lfloor \max(a_1,\ldots, a_N)\rfloor+1$.
\end{prop}
\begin{proof}
  Remark~\ref{H_N=T_N} ensures that $ \#(\lfloor A_N - A_N \rfloor)^+ =\#\lfloor(A_N -A_N)^{+}\rfloor(=H_N) $. Let $(Y_n)_{n \geq 1}$ be a sequence of distinct positive integers such that
    \begin{align}\label{Y_n equation}
        \{Y_n : 1 \leq n \leq H_N\} = (\lfloor A_N - A_N \rfloor)^+ \quad \text{for all } N.
    \end{align}
    Mimicking the proof of Theorem~\ref{upper bound} with a similar argument by replacing $(Z_n, H_N)$ with $(Y_n, H_N)$, we obtain
    \begin{align}
    \min_{1 \leq m \neq n \leq N} \left\| \alpha \left\lfloor a_m - a_n \right\rfloor \right\| =\displaystyle \min_{1\leq m\neq n\leq  N} \left \| \alpha |\left \lfloor a_{m}-a_{n} \right\rfloor| \right\|         = \min_{ t \leq H_N} \| \alpha Y_t \|  \leq   \frac{4\log _{2} \widetilde{B}_{N}}{H_{N} \log (\frac{N}{9}) \log _{2} (\frac{N}{9})}
\end{align} 
for infinitely many $N$.
\end{proof}
    
\begin{remark}
For any $x\in  \mathbb{R},$
\begin{align}
\lfloor -x \rfloor = 
\begin{cases} 
-\lfloor x \rfloor & \text{if } x \in \mathbb{Z}, \\
-\lfloor x \rfloor - 1 & \text{if } x \notin \mathbb{Z}.
\end{cases}
\end{align}
If $a_m-a_n \notin \mathbb{Z}$ for all $m \neq n$, we have
\begin{align}
      \min_{1\leq m\neq n\leq  N} \left \| \alpha \left (\lfloor a_{m}-a_{n} \right\rfloor +1)\right\|  
     = \min_{1\leq m\neq n\leq  N} \left \| \alpha \left (\lfloor a_{n}-a_{m} \right\rfloor +1)\right\|  
    = & \min_{1\leq m\neq n\leq  N} \left \| \alpha \left (-\lfloor a_{m}-a_{n} \right\rfloor)\right\|  \\
    =&\displaystyle \min_{1\leq m\neq n\leq  N} \left \| \alpha \left \lfloor a_{m}-a_{n} \right\rfloor \right\|    \\
    \leq &       \frac{4\log _{2} \widetilde{B}_{N}}{H_{N} \log (\frac{N}{9}) \log _{2} (\frac{N}{9})}
\end{align} 
for infinitely many $ N$.
\end{remark}
\begin{prop}
Let $(a_n)_{n \geq 1}$ be a sequence of distinct positive real numbers. For any integer $N \geq 2$, define
\begin{align}
 B_N = \max_{1 \leq m \neq n \leq N} \lfloor |a_m - a_n| \rfloor
\quad \text{and} \quad
    \widetilde{B}_N = \max_{1 \leq m \neq n \leq N} |\lfloor a_m - a_n \rfloor|.
\end{align}
Then $\widetilde{B}_N \geq B_N$ always holds.

Furthermore, the equality $\widetilde{B}_N = B_N$ is valid if and only if the maximum value $\widetilde{B}_N$ is achieved by a pair of indices $(m, n)$ for which the difference $|a_m - a_n|$ is an integer. In this case, $|\lfloor a_m - a_n \rfloor| = \lfloor |a_m - a_n| \rfloor$.
\end{prop}
\begin{proof}
Let  $m, n$ be distinct indices.

\begin{itemize}
    \item If $a_m \geq a_n$, then
    \begin{align}
            |\lfloor a_m - a_n \rfloor| = \lfloor a_m - a_n \rfloor = \lfloor |a_m - a_n| \rfloor.
    \end{align}
    \item If $a_m < a_n$, then using the identities $\lfloor -x \rfloor =-\lceil x  \rceil$ and $\lceil x \rceil \geq \lfloor x \rfloor$ for real $x$, we obtain
    \begin{align}
          |\lfloor a_m - a_n \rfloor| = | \lfloor -(a_n - a_m) \rfloor | = | -\lceil a_n - a_m \rceil | = \lceil a_n - a_m \rceil\geq \lfloor a_n - a_m \rfloor = \lfloor |a_m - a_n| \rfloor,
    \end{align}
   where the equality holds if and only if $a_n - a_m$ is an integer. Otherwise, it is a strict inequality.
\end{itemize}
Therefore, for each pair of distinct indices $m, n$, we have
$|\lfloor a_m - a_n \rfloor| \geq \lfloor |a_m - a_n| \rfloor$,
and the inequality is strict if $a_n - a_m$ is not an integer.

Taking the maximum over all pairs gives
\begin{align}
\widetilde{B}_{N} = \max_{1 \leq m \neq n \leq N} |\lfloor a_m - a_n \rfloor| \geq \max_{1 \leq m \neq n \leq N} \lfloor |a_m - a_n| \rfloor = B_N.
\end{align}
\end{proof}
\subsection{Proof of Theorem~\ref{upper bound independent of a_N}}
Theorem~\ref{upper bound} depends on the magnitude of $a_N$. To establish Theorem~\ref{upper bound independent of a_N} (that is, to obtain a bound independent of $a_N$), we modify the coprime approach of \eqref{Sn coprime} by removing only those subintervals centred at $a/Z_n$ such that $a$ and $Z_n$ share a small prime factor. This ensures that the removed measure depends solely on $N$, not $a_N$, while sufficiently reducing the problematic overlaps responsible for the growth of the GCD sum.

\subsubsection{Overlap estimates and auxiliary lemmas}\label{Overlap estimates and auxiliary lemmas}
Following these heuristic arguments, we now proceed to the formal proof. Let the sequence $\left(Z_{n}\right)_{n\geq1}$ be defined as in the previous section. For the remainder of this section, we restrict our indices $m$ and $n$ in a range $(2^{k / 2}, 2^{k}]$ for some positive integer $k$ (where, for simplicity of writing, we assume that $k$ is even). We set
$$
\psi(n)=\frac{1}{4 n}, \quad n \geq 1,
$$
and
$$
S_{n}^{*}=[0,1] \cap\left(\bigcup_{\substack{0 \leq a \leq Z_{n} \\ p \mid\left(a, Z_{n}\right) \Rightarrow p>4^{k}}}\left(\frac{a}{Z_{n}}-\frac{\psi(n)}{Z_{n}}, \frac{a}{Z_{n}}+\frac{\psi(n)}{Z_{n}}\right)\right).
$$
To establish Theorem~\ref{upper bound independent of a_N}, it suffices to prove that almost every $\alpha$ belongs to infinitely many $S_n^*$. We now state the necessary lemmas, providing proofs for those for which no reference is given.

\begin{lemma}\label{lemma 2}
For all $n \in(2^{k / 2}, 2^{k}],$
$$
\lambda(S_{n}^{*}) \geq \frac{1}{n} \frac{1}{e^{\gamma} \log \left(4^{k}\right)}(1+o(1))
$$
as $k \rightarrow \infty$, where $\gamma$ denotes the Euler-Mascheroni constant.
\end{lemma}

\begin{lemma}[Pollington-Vaughan style overlap estimate]\label{lemma 3} Let $m, n \in(2^{k / 2}, 2^{k}]$ with $m \neq n$. Set
\begin{align}
D(Z_{m}, Z_{n})=\frac{\max \left(Z_{m} \psi(n), Z_{n} \psi(m)\right)}{\gcd(Z_{m}, Z_{n})} .
\end{align}
Furthermore, when $D(Z_{m}, Z_{n}) \geq 1$, then set
\begin{align}
P(Z_{m}, Z_{n})=\prod_{\substack{p \mid \frac{Z_{m} Z_{n}}{(Z_{m}, Z_{n})^{2}}\\ D(Z_{m}, Z_{n})<p \leq 4^{k}}}\left(1+\frac{1}{p}\right), 
\end{align}
where the product ranges over all primes $p$ in the specified range. When $D(Z_{m}, Z_{n})<1$, then set $P(Z_{m}, Z_{n})=0$. Then
\begin{align}
\lambda(S_{m}^{*} \cap S_{n}^{*}) \ll \frac{\sqrt{\psi(m) \psi(n)}}{4^{k}}+P(Z_{m}, Z_{n}) \lambda\left(S_{m}^{*}\right) \lambda(S_{n}^{*}).
\end{align}
\end{lemma}

\begin{lemma}[Chung-Erd\H{o}s inequality]\label{lemma 4} Let $\mathcal{A}_{m}, 1 \leq m \leq M$, be events in a probability space $(\Omega, \mathcal{F}, \mathbb{P})$, such that $\mathbb{P}\left(\mathcal{A}_{m}\right)>0$ for at least one value of $m$. Then 
$$
\mathbb{P}\left(\bigcup_{m=1}^{M} \mathcal{A}_{m}\right) \geq \frac{\left(\displaystyle\sum_{m=1}^{M} \mathbb{P}\left(\mathcal{A}_{m}\right)\right)^{2}}{\displaystyle\sum_{1 \leq m, n \leq M} \mathbb{P}\left(\mathcal{A}_{m} \cap \mathcal{A}_{n}\right)}.
$$
\end{lemma}
\begin{proof}
See \cite[Theorem 1.4.3d]{chandra2012borel}.
\end{proof}

\begin{lemma}[Cassels' zero-one law \cite{Cassels_1950}]\label{lemma 5} Let $(\xi(m))_{m \geq 1}$ be a sequence of non-negative real numbers. For $m \geq 1$, set $\mathcal{A}_{m}=\displaystyle\bigcup_{0 \leq a \leq m}\left(\frac{a}{m}-\frac{\xi(m)}{m}, \frac{a}{m}+\frac{\xi(m)}{m}\right)$. Let $\mathcal{A}$ be the set of those $\alpha \in[0,1]$, which are contained in infinitely many $\mathcal{A}_{m}$. Then the Lebesgue measure of $\mathcal{A}$ is either 0 or 1.
\end{lemma}
Finally, we use the following  lemma, which played a crucial role in Koukoulopoulos and Maynard's proof of the Duffin-Schaeffer conjecture.

\begin{lemma}\label{lemma 6} Let $(\eta(q))_{q \geq 1}$ be a sequence of real numbers in $[0, 1/2]$. Set $M(q, r)=\max (q \eta(r), r \eta(q))$. Assume that there are $X<Y$ such that
\begin{align}\label{KM bound}
\sum_{X \leq q \leq Y} \frac{\eta(q) \varphi(q)}{q} \leq 2 . 
\end{align}
For $t \geq 1$, set
$$L_{t}(q, r)=\sum_{\substack{p \mid q r / \gcd(q, r)^{2} \\ p \geq t}} \frac{1}{p}$$
and
$$\mathcal{E}_{t}=\left\{(q, r) \in(\mathbb{Z} \cap[X, Y))^{2}: \gcd(q, r) \geq t^{-1} M(q, r) \text { and } L_{t}(q, r) \geq 10\right\}.
$$
Then
\begin{align}
\sum_{(q, r) \in \mathcal{E}_{t}} \frac{\eta(q) \varphi(q)}{q} \frac{\eta(r) \varphi(r)}{r} \ll \frac{1}{t}.
\end{align}
\end{lemma}
\begin{proof}
    See \cite[Proposition 5.4]{KoukouMaynard}.
\end{proof}
\begin{remark}
    As noted in \cite{ABM}, although the above lemma is formulated with the additional requirement that the sum in \eqref{KM bound} is at least $1$, this requirement is unnecessary and not used in the proof of \cite[Proposition 5.4]{KoukouMaynard}.
\end{remark}
We now prove Lemmas~\ref{lemma 2} and~\ref{lemma 3}.

\textbf{Proof of Lemma~\ref{lemma 2}.} Using $\psi(n)=1 /2n$, we have
$$
\begin{aligned}
\lambda(S_{n}^{*}) & =\frac{2\psi(n)}{Z_{n}} \#\left\{0 \leq a \leq Z_{n}: p \mid\left(a, Z_{n}\right) \Rightarrow p>4^{k}\right\} \\
& \geq \frac{1}{nZ_{n}} Z_{n} \prod_{p \leq 4^{k}}\left(1-\frac{1}{p}\right) \\
& =\frac{1}{n} \frac{1}{e^{\gamma} \log 4^{k}}(1+o(1)),
\end{aligned}
$$
where the product over primes is estimated using Mertens's third theorem \cite[Theorem 429]{Hardy}.

\textbf{Proof of Lemma~\ref{lemma 3}.} The following lemma adapts the overlap estimate of Pollington and Vaughan \cite{pollingtonvaughan1990}. For notational convenience, we use the following notation throughout this proof
$$
[\min ]=\min \left(\frac{\psi(m)}{Z_{m}}, \frac{\psi(n)}{Z_{n}}\right), \quad[\max ]=\max \left(\frac{\psi(m)}{Z_{m}}, \frac{\psi(n)}{Z_{n}}\right) .
$$
The Lebesgue measure of $S_{m}^{*} \cap S_{n}^{*}$ is bounded above by (see \cite[p.~39]{Harman1998})
$$ { [\min] } \sum_{\substack{1 \leq a \leq Z_{m} \\ p \mid\left(a, Z_{m}\right) \Rightarrow p>4^{k}}}\sum_{\substack{1 \leq a \leq Z_{n} \\ p \mid\left(b, Z_{n}\right) \Rightarrow p>4^{k}}}\mathds{1} \left(\left|\frac{a}{Z_{m}}-\frac{b}{Z_{n}}\right| \leq 2[\max ]\right).$$
Following Pollington and Vaughan's method, we bound the contribution from distinct fractions $a/Z_m \neq b/Z_n$, noting that we only sifted primes below $4^k$. Let $u_p=\nu_{p}(Z_m)$ and  $v_p=\nu_{p}(Z_n)$. Therefore, 
$$Z_{m}=\prod_{p} p^{u_p} \quad \text { and } \quad Z_{n}=\prod_{p} p^{v_p}.
$$
Let
$$
f=\prod_{\substack{p \\ u_p=v_p}} p^{u_p}, \quad g=\prod_{\substack{p \\ u_p \neq v_p}} p^{\min (u_p, v_p)}, \quad \text { and } \quad h=\prod_{\substack{p \\ u_p \neq v_p}} p^{\max (u_p, v_p)} .
$$

Note that
$$\frac{h}{g}=\prod_{\substack{p \\ u_p \neq v_p}} p^{\max (u_p, v_p)-\min (u_p, v_p)}=\frac{\prod_{p} p^{\max (u_p, v_p)}}{\prod_{p} p^{\min (u_p, v_p)}}=\frac{\operatorname{lcm}(Z_{m}, Z_{n})}{\gcd(Z_{m}, Z_{n})}=\frac{Z_{m} Z_{n}}{\gcd(Z_{m}, Z_{n})^{2}}.$$
We now use the representation
$$D(Z_{m}, Z_{n})=[\max ] f h$$
to obtain
$$[\min ] g f D(Z_{m}, Z_{n})=\psi(m) \psi(n).$$

Therefore, as mentioned in \cite{ABM}, an adaptation of the Pollington-Vaughan method yields

\begin{align}\label{sieve expression}
& \text { [min] } \sum_{\substack{1 \leq a \leq Z_{m} \\ p \mid\left(a, Z_{m}\right) \Rightarrow p>4^{k}}} \sum_{\substack{1 \leq b \leq Z_{n} \\ p \mid\left(b, Z_{n}\right) \Rightarrow p>4^{k}}}  \mathds{1}
\left(0<\left|\frac{a}{Z_{m}}-\frac{b}{Z_{n}}\right| \leq 2[\max ]\right)\\
\ll &[\min ] g \prod_{\substack{p \mid g\\ p \leq 4^{k}}}\left(1-\frac{1}{p}\right) f \prod_{\substack{p \mid f\\ p \leq 4^{k}}}\left(1-\frac{1}{p}\right)^{2} \sum_{\substack{0<j \leq 2 D(Z_{m}, Z_{n})\\ p \mid(j, h) \Rightarrow p>4^{k}}} \prod_{\substack{p \mid(j, f)\\ p \leq 4^{k}}}\left(1+\frac{1}{p}\right) \\
 \ll & \lambda\left(S_{m}^{*}\right) \lambda(S_{n}^{*}) P(Z_{m}, Z_{n}),
\end{align}
where the sum over $j$ is estimated by Brun's sieve, mirroring the technique employed by Pollington and Vaughan as in \cite[p.~195--196]{pollingtonvaughan1990}.

In the Duffin-Schaeffer framework of Pollington and Vaughan, the complete co-primality condition ensures that there is no contribution to the overlap $S_{m}^{*} \cap S_{n}^{*}$ from pairs $(a, b)$ with $a / Z_{m}=b / Z_{n}$. In our setting, these overlaps contribute
\begin{align}\label{Equality on a/z_n=b/Z_m} 
\ll[\mathrm{min}] \sum_{\substack{0 \leq a \leq Z_{m} \\ p \mid\left(a, Z_{m}\right) \Rightarrow p>4^{k}}} \sum_{\substack{0 \leq b \leq Z_{n} \\ p \mid\left(b, Z_{n}\right) \Rightarrow p>4^{k}}}  \mathds{1}\left(\frac{a}{Z_{m}}=\frac{b}{Z_{n}}\right) . 
\end{align}
We have $a / Z_{m}=b / Z_{n}$ if and only if
$$
a=j \frac{Z_{m}}{\gcd(Z_{m}, Z_{n})} \quad \text { and } \quad b=j \frac{Z_{n}}{\gcd(Z_{m}, Z_{n})}
$$
for some $j$ in $0 \leq j \leq \gcd(Z_{m}, Z_{n})$. We claim that for all $Z_{m}$ and $Z_{n}$ for which there exists $p \leq 4^{k}$ such that the $p$-adic valuation of $Z_{m}$ is different from that of $Z_{n}$, there is no admissible value of $j$, and the double sum above is empty. 
In other words, a non-vanishing contribution requires that $\nu_{p}(Z_m) = \nu_{p}(Z_n)$ for all $p \leq 4^{k}$. However, if $\nu_{p}(Z_m) = \nu_{p}(Z_n)$ for all $p > 4^{k}$ as well, then $Z_m = Z_n$, which is a contradiction. Therefore, a non-vanishing contribution implies that there exists a prime $p > 4^k$ such that $\nu_{p}(Z_m) \neq \nu_{p}(Z_n)$.

We now prove the claim mentioned above by contrapositive. Suppose that there exists an admissible value of $j$, which means that there exist $a$ and $b$ satisfying $0\leq a\leq Z_m$, $p\mid (a,Z_m) \implies p >4^{k}$, and  $0\leq b\leq Z_n$, $p\mid (b,Z_n) \implies p >4^{k}$ such that $\frac{a}{Z_m}=\frac{b}{Z_n}$, that is, $aZ_n=bZ_m$. This shows that $\nu_{p}(Z_m)=\nu_{p}(Z_n)$ for all $p \leq 4^{k}$. If not, there exists $q\leq 4^k$ such that $\nu_{q}(Z_m)\neq\nu_{q}(Z_n)$. Without loss of generality, assume $\nu_{q}(Z_n) >\nu_{q}(Z_m) \geq 0$, that is, $\nu_{q}(Z_n) -\nu_{q}(Z_m) \geq 1$. Then $q$ divides $a q^{\nu_{q}(Z_n) -\nu_{q}(Z_m)} \displaystyle\prod_{ p\neq q} p^{\nu_{p} (Z_n)}=b \displaystyle\prod_{p\neq q} p^{\nu_{p} (Z_m)}$, that is, $q$ divides $b$. In addition, $\nu_{q}(Z_n) \geq 1$ implies $q|Z_n$, so $q|(b,Z_n)$  with $q\leq 4^k$, which is a contradiction. This contradiction shows that $\nu_{p}(Z_m)=\nu_{p}(Z_n)$ for all $p \leq 4^{k}$.

Since a non-vanishing contribution implies the existence of a prime $p > 4^k$ such that $\nu_{p}(Z_m) \neq \nu_{p}(Z_n)$, we have $\gcd(Z_{m}, Z_{n}) \leq \min (Z_{m}, Z_{n}) / 4^{k}$. Thus, we can estimate \eqref{Equality on a/z_n=b/Z_m} by
$$\ll \min \left(\frac{\psi(m)}{Z_{m}}, \frac{\psi(n)}{Z_{n}}\right) \frac{\min (Z_{m}, Z_{n})}{4^{k}} \ll \frac{\sqrt{\psi(m) \psi(n)}}{\sqrt{Z_{m} Z_{n}}} \frac{\sqrt{Z_{m} Z_{n}}}{4^{k}} \ll \frac{\sqrt{\psi(m) \psi(n)}}{4^{k}} .$$
This together with \eqref{sieve expression} proves the lemma.

\subsubsection{Proof of Theorem~\ref{upper bound independent of a_N}, upper bound independent of the size of $a_{N}$.}\label{section upper bound independent of aN}
Assume that $k$ is a fixed, even, and ``large'' number. We decompose all numbers $Z_{n}$ as follows.
$$
Z_{n}=Z_{n}^{\text {smooth }} \cdot Z_{n}^{\text {rough }}, \quad 2^{k / 2}<n \leq 2^{k},
$$
where $Z_{n}^{\text {smooth }}$ has only prime factors of size at most $4^{k}$, and $Z_{n}^{\text {rough }}$ has only prime factors of size larger than $4^{k}$, that is, $Z_{n}^{\text {smooth }} = \displaystyle\prod_{p \leq 4^k}p^{\nu_{p}(Z_n)}$ and $Z_{n}^{\text {rough }} = \displaystyle\prod_{p > 4^k}p^{\nu_{p}(Z_n)}$.  Write
$$
\left\{b_{1}, \ldots, b_{H}\right\}=\left\{Z_{n}^{\text {rough }}, 2^{k / 2}<n \leq 2^{k}\right\}
$$

for some appropriate $H$, where we assume that $b_{1}, \ldots, b_{H}$ are sorted in increasing order. Clearly, $H \leq 2^{k}-2^{k / 2}$, but $H$ might actually be smaller since we could have $Z_{m}^{\text {rough }}=Z_{n}^{\text {rough }}$ for some $m \neq n$. For every $n \in (2^{k / 2}, 2^{k}]$, we now define
$$ Y_{n} \coloneqq Z_{n}^{\text {smooth }} p_{h}^{(k)}, $$
where $h$ is the uniquely defined index for which $Z_{n}^{\text {rough }}=b_{h}$, and where $p_{1}^{(k)}$ denotes the smallest prime exceeding $4^{k}, p_{2}^{(k)}$ denotes the second-smallest prime exceeding $4^{k}$, and so on. The point in the construction of the numbers $\left(Y_{n}\right)_{2^{k / 2}<n \leq 2^{k}}$ is that on the one hand, in Lemma~\ref{lemma 3}, we can replace $P(Z_{m}, Z_{n})$ and $D(Z_{m}, Z_{n})$ by $P\left(Y_{m}, Y_{n}\right)$ and $D\left(Y_{m}, Y_{n}\right)$ in the relevant situations, since the small prime factors of $Z_{m}$ and $Z_{n}$ are the same as those of $Y_{m}$ and $Y_{n}$, respectively. On the other hand, we have
$$\prod_{p \mid Y_{n}}\left(1-\frac{1}{p}\right) \leq \prod_{\substack{p \mid Z_{n}\\ p \leq 4^{k}}}\left(1-\frac{1}{p}\right) \leq\left(1-\frac{1}{4^{k}}\right)^{-1} \prod_{p \mid Y_{n}}\left(1-\frac{1}{p}\right),$$
where we used that by construction the ``small'' prime factors of $Z_{n}$ and $Y_{n}$ coincide, and $Y_{n}$ has an additional prime factor that exceeds $4^{k}$. Consequently,
\begin{equation}\label{bounding measure of S_n}
\frac{2 \psi(n) \varphi\left(Y_{n}\right)}{Y_{n}} \leq \lambda(S_{n}^{*}) \leq \frac{2 \psi(n) \varphi\left(Y_{n}\right)}{Y_{n}}\left(1-\frac{1}{4^{k}}\right)^{-1} . 
\end{equation}
Thus, we can control the size of the Euler totient function of $Y_{n}$ (while we cannot control it for $Z_{n}$, which might have many large prime factors). This will allow us to apply Lemma~\ref{lemma 6}. By Lemma~\ref{lemma 4}, we have
\begin{equation}\label{chung erdos inequality}
\lambda\Big(\bigcup_{2^{k / 2}<n \leq 2^{k}} S_{n}^{*}\Big) \geq \frac{\Big(\displaystyle\sum_{2^{k / 2}<n \leq 2^{k}} \lambda(S_{n}^{*})\Big)^{2}}{\displaystyle\sum_{2^{k / 2}<m, n \leq 2^{k}} \lambda(S_{m}^{*} \cap S_{n}^{*})} . 
\end{equation}

By Lemma~\ref{lemma 2}, we have
\begin{equation}\label{lower bounding measure of S_n}
\sum_{2^{k / 2}<n \leq 2^{k}} \lambda(S_{n}^{*}) \geq \sum_{2^{k / 2}<n \leq 2^{k}} \frac{1}{n} \frac{1}{e^{\gamma} \log 4^{k}}(1+o(1)) \geq 0.14
\end{equation}
for sufficiently large $k$. On the other hand, we can assume without loss of generality that
\begin{equation}\label{upper bounding measure of S_n}
\sum_{2^{k / 2}<n \leq 2^{k}} \lambda(S_{n}^{*}) \leq 0.99 
\end{equation}

(if the sum of measures is even larger, we can just delete some of the sets $S_{n}^{*}$ ). Thus, we can control the size of the numerator on the right-hand side of \eqref{chung erdos inequality}. To estimate the denominator of the right-hand side of \eqref{chung erdos inequality}, by Lemma~\ref{lemma 3}, we have
\begin{equation}\label{intersection of Sm and Sn}
\lambda(S_{m}^{*} \cap S_{n}^{*}) \ll \frac{\sqrt{\psi(m) \psi(n)}}{4^{k}}+P(Z_{m}, Z_{n}) \lambda\left(S_{m}^{*}\right) \lambda(S_{n}^{*}) . 
\end{equation}
Trivially,
\begin{equation}\label{bound for sqrt psi(m) psi(n)}
\sum_{2^{k / 2}<m, n \leq 2^{k}} \frac{\sqrt{\psi(m) \psi(n)}}{4^{k}} \ll 1 . \
\end{equation}
Note that whenever $Z_{m}^{\text {rough }} \neq Z_{n}^{\text {rough }}$, we have $\gcd(Z_{m}, Z_{n}) \leq 4^{-k} \min (Z_{m}, Z_{n})$, and so
$$
D(Z_{m}, Z_{n})=\frac{\max \left(Z_{m} \psi(n), Z_{n} \psi(m)\right)}{\gcd(Z_{m}, Z_{n})} \geq \frac{4^{k} \max (Z_{m}, Z_{n})}{2^{k} \min (Z_{m}, Z_{n})} \geq 2^{k}.
$$
Therefore, in this case
$$
P(Z_{m}, Z_{n}) \ll \prod_{2^{k} \leq p \leq 4^{k}}\left(1+\frac{1}{p}\right) \ll 1 .
$$

On the other hand, if $Z_{m}^{\text {rough }}=Z_{n}^{\text {rough }}$, then it is easily seen that $D\left(Y_{m}, Y_{n}\right)=D(Z_{m}, Z_{n})$. In the next displayed formula, all sums are taken over $m, n$ in the range $2^{k / 2}<m, n \leq 2^{k}$. Using \eqref{bounding measure of S_n}, we can estimate
\begin{align}\label{sum P(zm,zn)}
& \sum_{m, n} P(Z_{m}, Z_{n}) \lambda\left(S_{m}^{*}\right) \lambda(S_{n}^{*}) \\
 \ll &  \sum_{\substack{m, n\\
Z_{m}^{\text {rough }} \neq Z_{n}^{\text {rough }}}} \lambda\left(S_{m}^{*}\right) \lambda(S_{n}^{*})+\sum_{\substack{m, n \\
Z_{m}^{\text {rough }}=Z_{n}^{\text {rough }}}} \lambda\left(S_{m}^{*}\right) \lambda(S_{n}^{*})\prod_{\substack{p \mid \frac{Z_{m} Z_{n}}{(Z_{m}, Z_{n})^{2}}\\ D(Z_{m}, Z_{n})<p \leq 4^{k}}}\left(1+\frac{1}{p}\right)\\
 \ll &  \sum_{m, n} \lambda\left(S_{m}^{*}\right) \lambda(S_{n}^{*})+\sum_{m, n} \frac{\psi(m) \varphi\left(Y_{m}\right)}{Y_{m}} \frac{\psi(n) \varphi\left(Y_{n}\right)}{Y_{n}} \prod_{\substack{p \mid \frac{Y_{m} Y_{n}}{\left(Y_{m}, Y_{n}\right)^{2}}\\ D\left(Y_{m}, Y_{n}\right)<p \leq 4^{k}}}\left(1+\frac{1}{p}\right) \\
 \ll &  \sum_{m, n} \lambda\left(S_{m}^{*}\right) \lambda(S_{n}^{*})+\sum_{t=0}^{k-1} \sum_{\substack{2^{k / 2}<m, n \leq 2^{k} \\
D\left(Y_{m}, Y_{n}\right) \in[4^{t}, 4^{t+1})}} \frac{\psi(m) \varphi\left(Y_{m}\right)}{Y_{m}} \frac{\psi(n) \varphi\left(Y_{n}\right)}{Y_{n}} \prod_{\substack{p \mid \frac{Y_{m} Y_{n}}{\left(Y_{m}, Y_{n}\right)^{2}}\\ p > 4^{t+1}}}\left(1+\frac{1}{p}\right).
\end{align}

We divide the second sum in the last line of \eqref{sum P(zm,zn)} depending on whether $\left(Y_{m}, Y_{n}\right)$ is in $\mathcal{E}_{4^{t+1}}$ or not. Bounding the Euler product trivially, we obtain
\begin{align}
& \sum_{t=0}^{k-1} \sum_{\substack{2^{k / 2}<m, n \leq 2^{k}, \\D\left(Y_{m}, Y_{n}\right) \in [4^{t},4^{t+1}),\\\left(Y_{m}, Y_{n}\right) \notin \mathcal{E}_{{4}^{t+1}}    }} \frac{\psi(m) \varphi\left(Y_{m}\right)}{Y_{m}} \frac{\psi(n) \varphi\left(Y_{n}\right)}{Y_{n}} \prod_{\substack{p \mid \frac{Y_{m} Y_{n}}{\left(Y_{m}, Y_{n}\right)^{2}}\\ p > 4^{t+1}}}\left(1+\frac{1}{p}\right)\\
\ll & \sum_{t=0}^{k-1} \sum_{\substack{2^{k / 2}<m, n \leq 2^{k}, \\D\left(Y_{m}, Y_{n}\right) \in[4^{t},4^{t+1})}} \frac{\psi(m) \varphi\left(Y_{m}\right)}{Y_{m}} \frac{\psi(n) \varphi\left(Y_{n}\right)}{Y_{n}} \\
\ll & \sum_{m, n} \lambda\left(S_{m}^{*}\right) \lambda(S_{n}^{*}) .
\end{align}

In the case where $\left(Y_{m}, Y_{n}\right) \in \mathcal{E}_{4^{t+1}}$, we apply Lemma~\ref{lemma 6} with
\begin{align}
        \eta(q)=\begin{cases}
   \psi(n) &\quad \text {if } q =Y_{n} \text{ for some } n  \in (2^{k / 2}, 2^{k}],\\
   0  &\quad\text {otherwise,} 
   \end{cases}
\end{align}
and with $X$ and $Y$ defined as the minimum and maximum, respectively, of the set $ \{Y_n: 2^{k/2} < n \leq 2^{k}\}$. Then by \eqref{lower bounding measure of S_n}, and \eqref{upper bounding measure of S_n}, we have
\begin{align}
\sum_{X \leq q \leq Y} \frac{\eta(q) \varphi(q)}{q}=\sum_{2^{k / 2}<n \leq 2^{k}} \frac{\psi(n) \varphi\left(Y_{n}\right)}{Y_{n}} \in[1 / 16,1 / 2] ,
\end{align}
for sufficiently large $k$.

In the same way as Koukoulopoulos and Maynard deduce \cite[Theorem 1]{KoukouMaynard} from Lemma~\ref{lemma 6}, we obtain
\begin{align}
& \sum_{t=0}^{k-1} \sum_{\substack{2^{k / 2}<m, n \leq \leq 2^{k}, \\D\left(Y_{m}, Y_{n}\right) \leq 4^{t+1},\\
(Y_m, Y_n) \in \mathcal{E}_{4} t+1}} \frac{\psi(m) \varphi\left(Y_{m}\right)}{Y_{m}} \frac{\psi(n) \varphi\left(Y_{n}\right)}{Y_{n}} \prod_{\substack{p \mid \frac{Y_{m} Y_{n}}{\left(Y_{m}, Y_{n}\right)^{2}}\\ p > 4^{t+1}}}\left(1+\frac{1}{p}\right) 
\ll  \sum_{t=0}^{k-1} \frac{t}{4^{t}} \ll  1 .
\end{align}
Combining \eqref{intersection of Sm and Sn}, \eqref{bound for sqrt psi(m) psi(n)}, and \eqref{sum P(zm,zn)}, we obtain
$$
\sum_{2^{k / 2} \leq m, n \leq 2^{k}} \lambda(S_{m}^{*} \cap S_{n}^{*}) \ll 1 .
$$
Thus, by \eqref{chung erdos inequality} and \eqref{lower bounding measure of S_n}, for all sufficiently large $k$, 
$$
\lambda\left(\bigcup_{2^{k / 2}<n \leq 2^{k}} S_{n}^{*}\right) \gg 1
$$
where the implied constant is independent of $k$. Since $k$ can be chosen arbitrarily large, this implies
$$
\lambda\left(\bigcap_{\ell=1}^{\infty}\left(\bigcup_{n=\ell}^{\infty} S_{n}^{*}\right)\right)>0,
$$
and since $S_{n}^{*} \subset S_{n}$, we clearly also have
$$
\lambda\left(\bigcap_{\ell=1}^{\infty}\left(\bigcup_{n=\ell}^{\infty} S_{n}\right)\right)>0 .
$$
Thus, the measure of the limsup set is positive, which, by Cassels' zero-one law (Lemma~\ref{lemma 5}), implies that it is indeed $1$. In other words, almost all $\alpha \in[0,1]$ are contained in infinitely many sets $S_{n}$. Thus, for almost all $\alpha$, there are infinitely many $n$ such that
$$
\left\|\alpha Z_{n} \right\| \leq \frac{1}{4 n} .
$$
By the construction of $Z_n$, for every $n$, there is a unique $N$ such that $H_{N-1}< n \leq H_{N}$, and Lemma~\ref{H_N bound for general sequence} shows that $n > H_{N-1} \geq \frac{1}{4}H_{N}$. Therefore, for all $n \in \mathbb{N}$ with $n \in (H_{N-1},H_{N}]$, we have
$$  \left\|  \alpha Z_{n}\right\|\leq \frac{1}{4n} \leq \frac{1}{H_{N}},$$
which implies that
\begin{align}
   \lfloor \delta^{\alpha}_{\min }\rfloor(N)=&\displaystyle \min_{1\leq m\neq n\leq  N}\| \alpha\lfloor |a_{m}-a_{n}| \rfloor\| 
   = \displaystyle \min_{t\leq  H_{N}}\|\alpha Z_{t}\|
   \leq \|\alpha Z_{n}\| \leq \frac{1}{H_{N} }.
\end{align}
Thus,
\begin{align}
\lfloor \delta^{\alpha}_{\min }\rfloor(N) \leq \frac{1}{H_{N}}
\end{align}
for infinitely many $N$ and for almost all $\alpha$, as required.

\subsection{Proof of Theorem~\ref{upper bound in terms of H_N for sufficiently large N}}\label{Subsection:upper bound in terms of H_N for sufficiently large N}
To prove Theorem~\ref{upper bound in terms of H_N for sufficiently large N}, we use the methodology developed by Rudnick~\cite{Rudnick}. Let $A_N$ be defined as in Section~\ref{Notation}. Denote the set of floored positive differences by $\mathcal{Z}_N$, that is, $\mathcal{Z}_N = \lfloor(A_N - A_N)^+ \rfloor=\{Z_1, \dots, Z_{H_N}\}$. Therefore, $|\mathcal{Z}_N|=H_N$ and the hypothesis ensure that $\mathcal{Z}_N \subseteq \mathbb{N}$.

        Let $f \in C_{c}^\infty(\R)$ be a non-negative, smooth, even function supported in $[-1/2, 1/2]$ with $\int_{-\infty}^\infty f(x) dx = 1$. Define the function
\begin{align}
F_M(x) = \sum_{j \in \Z} f(M(x + j)),    
\end{align}
which is $1$-periodic and is localised on the scale $1/M$.

We then define the counting function
\begin{align}\label{D(N,M)(alpha) definition}
D(N,M)(\alpha) = \sum_{z \in \mathcal{Z}_N} F_M(\alpha z).
\end{align}
Since
\begin{align}
\int_{0}^{1}    F_M(\alpha z)\,d\alpha= \sum_{j \in \Z} \int_{0}^{1} f(M(\alpha z + j))\,d\alpha=\frac{1}{z} \sum_{j \in \Z} \int_{j}^{j+z}f(My)\,dy =\int_{-\infty}^{\infty}f(My)\,dy=\frac{1}{M},
\end{align}
the expected value of $D(N,M)$ is
\begin{align}
    \mathbb{E}(D(N,M))= \int_{0}^{1}  D(N,M)(\alpha) \,d\alpha= \sum_{z \in \mathcal{Z}_N}\int_{0}^{1}    F_M(\alpha z)\,d\alpha=\frac{|\mathcal{Z}_N|}{M}=\frac{H_N}{M}.
\end{align}
The Fourier expansion of $F_M$ is given by
\begin{align}
F_M(x) = \sum_{k \in \Z} \frac{1}{M} \widehat{f}\left(\frac{k}{M}\right) e(kx),    
\end{align}
where $\widehat{f}(y) = \int_{-\infty}^\infty f(x) e^{-2\pi i xy}\, dx$ and $e(x) = e^{2\pi i x}$. Putting this into the definition of $D(N,M)$ yields
\begin{align}
    D(N,M)(\alpha) = \sum_{z \in \mathcal{Z}_N} F_M(\alpha z)=\sum_{k \in \Z} \frac{1}{M} \widehat{f}\left(\frac{k}{M}\right) \sum_{z \in \mathcal{Z}_N}e(k\alpha z).
\end{align}
Since $ \frac{\widehat{f}(0)}{M} |\mathcal{Z}_N| = \frac{H_N}{M}$, The expected value comes from the term corresponding to $k = 0$. Therefore, the variance, denoted by $\text{Var }D(N,M)$, is the second moment of the sum over non-zero frequencies:
\begin{align}
\text{Var }(D(N,M)) =& \int_0^1 |D(N,M)(\alpha) - \mathbb{E}[D(N,M)]|^2 d\alpha\\
=& \int_0^1 \Big| \sum_{ 0\neq k \in \Z} \frac{1}{M} \widehat{f}\left(\frac{k}{M}\right) \sum_{z \in \mathcal{Z}_N} e(k\alpha z) \Big|^2 d\alpha \\
=& \sum_{k_1,k_2 \neq 0} \frac{1}{M^2} \widehat{f}\left(\frac{k_1}{M}\right) \widehat{f}\left(\frac{k_2}{M}\right) \sum_{v_1,v_2 \in \mathcal{Z}_N} \int_0^1 e(\alpha(k_1 v_1 - k_2 v_2)) d\alpha.
\end{align}
The integral equals $1$ if $k_1v_1 - k_2v_2 = 0$ and $0$ otherwise, so
$$\int_0^1 e(\alpha(k_1 v_1 - k_2 v_2)) d\alpha=\mathds{1}_{\{0\}}(k_1v_1-k_2v_2).$$

As $(Z_n)_{n \geq 1}$ consists of distinct integers, we apply the GCD sum estimate from \cite{Bondarenko-Seip} to obtain
\begin{align}
    \sum_{m,n=1}^{H_N} \frac{\gcd(Z_m, Z_n)}{\sqrt{Z_m Z_n}} \ll H_N\exp \left(   A\sqrt{\frac{\log H_N \log_{3}H_N}{\log_{2}H_N}}\right)
\end{align}
for some absolute constant $A$ less than $7$. Combined with \cite[Lemma 4]{Rudnick}, this yields
\begin{align}\label{eq:variance bound}
\text{Var }(D(N,M)) 
=&\sum_{v_1,v_2 \in \mathcal{Z}_N} \sum_{k_1,k_2 \neq 0} \frac{1}{M^2} \widehat{f}\left(\frac{k_1}{M}\right) \widehat{f}\left(\frac{k_2}{M}\right) \mathds{1}_{\{0\}}(k_1v_1-k_2v_2)\\
\ll & \frac{1}{M} \sum_{v_1,v_2 \in \mathcal{Z}_N}  \frac{\gcd(v_1, v_2)}{\sqrt{v_1 v_2}}\\
\ll & \frac{1}{M} \sum_{1\leq k,l \leq H_N}  \frac{\gcd(Z_k, Z_l)}{\sqrt{Z_k Z_l}}\\
\ll &  \frac{H_N}{M} \exp \left(   A\sqrt{\frac{\log H_N \log_{3} H_N}{\log_{2} H_N}}\right)\\
\ll&   \frac{H_N}{M} \exp \left(   A\sqrt{\frac{\log N \log_{3}N}{\log_{2}N}}\right)
\ll   \frac{H_N}{M} \exp \left(  \epsilon \log N\right)=\frac{H_N N^\epsilon}{M}
\end{align}
for any $\epsilon >0$.

\textbf{Proof of Theorem~\ref{upper bound in terms of H_N for sufficiently large N} (upper bound for all except finitely many $N$).}

The proof is carried out in three steps. First, we establish a relationship between $H_{N_{k+1}}$ and $H_{N_k}$ along some subsequence $N_k$. In the next step, we show that the statistic $D(N,M)$ concentrates around its mean for an appropriate choice of $M$. Then, we use this concentration to deduce the existence of small gaps.

\begin{lemma}\label{relation between H_N_k}
Let $\eta \in (0,2]$ be a fixed parameter and $N_k = \lfloor k^{\eta/2} \rfloor$. Then there exists a constant $C > 0$ such that for all $k \geq 1$,
\begin{align}
H_{N_{k+1}} \leq C_{\eta} H_{N_k},
\end{align}
where $C_{\eta}=5 + 2\eta$. In particular, one can take $C_{\eta}=9$.
\end{lemma}

\begin{proof}
With $A_N$ defined in Section~\ref{Notation}, we have $A_{N_{k+1}} = A_{N_k} \cup B$, where $B$ contains $N_{k+1} - N_k$ new elements.

Each new element $x \in B$ creates at most $2N_k$ new positive differences:
\begin{itemize}
    \item $\lfloor x - y\rfloor $ for $y \in A_{N_k}$ (positive if $x > y$).
    \item $\lfloor y - x\rfloor $ for $y \in A_{N_k}$ (positive if $y > x$).
\end{itemize}
Thus,
\begin{align}\label{H_N_k recurrence}
H_{N_{k+1}} \leq H_{N_k} +  2N_k (N_{k+1} - N_k) .     
\end{align}
Since $N_k = \lfloor k^{\eta/2} \rfloor$, we have
\begin{align}
N_{k+1} - N_k \leq (k+1)^{\eta/2} - (k^{\eta/2} - 1) \leq 1 + [(k+1)^{\eta/2} - k^{\eta/2}].    
\end{align}
We now apply the mean value theorem to $g(x) = x^{\eta/2}$ on $[k, k+1]$ to obtain
\begin{align}
    (k+1)^{\eta/2} - k^{\eta/2} = \frac{\eta}{2} c^{\eta/2 - 1} \quad \text{for some } c \in (k, k+1).
\end{align}
Since $c > k$ and $\eta \leq 2$ (so $\eta/2 - 1 \leq 0$), we have $c^{\eta/2 - 1} \leq k^{\eta/2 - 1}$. Therefore,
\begin{align}
    (k+1)^{\eta/2} - k^{\eta/2} \leq \frac{\eta}{2} k^{\eta/2 - 1},
\end{align}
and thus
\begin{align}\label{N_k recuurence}
    N_{k+1} - N_k \leq 1 + \frac{\eta}{2} k^{\eta/2 - 1}. 
\end{align}
From \eqref{H_N_k recurrence} and the lower bound $H_{N_k} \geq N_k/2$, we get
\begin{align}
\frac{H_{N_{k+1}}}{H_{N_k}} \leq 1 + \frac{2N_k(N_{k+1} - N_k)}{H_{N_k}} \leq 1 + \frac{2N_k(N_{k+1} - N_k)}{N_k/2} = 1 + 4(N_{k+1} - N_k) \leq 5 + 2\eta \, k^{\eta/2 - 1}.    
\end{align}
Since $\eta \leq 2$, we have $k^{\eta/2 - 1} \leq 1$ for all $k \geq 1$. Therefore,
\begin{align}
\frac{H_{N_{k+1}}}{H_{N_k}} \leq 5 + 2\eta.    
\end{align}
\end{proof}

\begin{lemma} \label{lemma:concentration}
Let $\eta \in (0,2]$, and define $N_k = \lfloor k^{2/\eta} \rfloor$ and $M_k = H_{N_k}/N_k^\eta$, where $N_k$ is defined as in Lemma~\ref{relation between H_N_k}. Then, for almost every $\alpha \in [0,1]$,
\begin{align}
D(N_k, M_k)(\alpha) \sim \frac{H_{N_k}}{M_k}.
\end{align}
\end{lemma}
\begin{proof}
Let $\delta > 0$ be arbitrary. By Chebyshev's inequality and the variance bound~\eqref{eq:variance bound}, we have
\begin{align}
\mathbb{P}\left( \left| D(N_k, M_k)(\alpha) - \frac{H_{N_k}}{M_k} \right| > \delta \cdot \frac{H_{N_k}}{M_k} \right) 
&\leq \frac{\text{Var } D(N_k, M_k)}{\delta^2 \left( \frac{H_{N_k}}{M_k} \right)^2} 
\ll \frac{\frac{H_{N_k}}{M_k} N_k^\epsilon}{\delta^2 \frac{H_{N_k}^2}{M_k^2}} 
= \frac{M_k N_k^\epsilon}{\delta^2 H_{N_k}}
\end{align}
for any $\epsilon >0$. Substituting $M_k = H_{N_k}/N_k^\eta$ gives
\begin{align}
\mathbb{P}\left( \left| \frac{D(N_k, M_k)}{H_{N_k}/M_k} - 1 \right| > \delta \right) \ll \frac{1}{\delta^2 N_k^{\eta - \epsilon}}.    
\end{align}
Choose $0 < \epsilon < \frac{\eta}{2}$, so that $\eta - \epsilon > \frac{\eta}{2} > 0$. Since  $\lfloor x \rfloor \geq \frac{x}{2}$ for $x>2$, we have $N_k \geq \frac{k^{2/\eta}}{2}$ for large $k$.
With the subsequence $N_k = \lfloor k^{2/\eta} \rfloor$, we then have
\begin{align}
\sum_{k=1}^\infty \frac{1}{N_k^{\eta - \epsilon}} < \infty.    
\end{align}
By the Borel–Cantelli lemma, for almost all $\alpha$, there exists $k_0(\alpha)$ such that
\begin{align}
\left| \frac{D(N_k, M_k)(\alpha)}{H_{N_k}/M_k} - 1 \right| \leq \delta \quad \text{for all } k \geq k_0(\alpha),    
\end{align}
which establishes the corollary.
\end{proof}
Lemma~\ref{lemma:concentration} implies that for almost all $\alpha$, there exists $k_0(\alpha)$ such that for all $k \geq k_0(\alpha)$,
\begin{align}
D(N_k, M_k)(\alpha) \sim \frac{H_{N_k}}{M_k} = N_k^{\eta}.    
\end{align}
In particular, since $N_k^{\eta} \to \infty$ as $k \to \infty$, we have $D(N_k, M_k)(\alpha) > 1$ for all sufficiently large $k$.

We have shown that for almost all $\alpha$, there exists $k_0(\alpha)$ such that $D(N_k, M_k)(\alpha) > 1$ for all $k \geq k_0(\alpha)$. It then follows from the definition of $D(N, M)(\alpha)$ (see~\eqref{D(N,M)(alpha) definition}) that for each $k \geq k_0(\alpha)$, there exists some $z \in \mathcal{Z}_{N_k}$ such that $F_{M_k}(\alpha z)>0$, that is, there exists $j_0 \in \Z$ such that $f(M_k(\alpha z+j_0))>0$.

Since $f$ is supported in $[-1/2, 1/2]$, we have $|\alpha z + j_0| \leq \frac{1}{2M_k}$ for some integer $j_0$, which implies $\|\alpha z\| \leq \frac{1}{M_k}$ for some $z \in \mathcal{Z}_{N_k}$ and all $k \geq k_0(\alpha)$. Therefore, for almost every $\alpha$, there exists $k_0(\alpha)$ such that
\begin{align}
     \lfloor \delta_{\min}^{\alpha} \rfloor(N_k) = \displaystyle \min_{\widetilde {z}\in \mathcal{Z}_{N_k} }\|\alpha \widetilde {z}\| \leq \|\alpha z\|  \leq \frac{1}{M_k} = \frac{N_k^\epsilon}{H_{N_k}}
\end{align}
for all $k \geq k_{0}(\alpha)$ and any $\epsilon >0$.

Let $N \geq \lfloor k_{0}(\alpha)^{2/\eta }\rfloor$. Then $N \in \big[ \lfloor (k_{0}(\alpha)+m)^{2/\eta}\rfloor, \lfloor (k_{0}(\alpha)+m+1)^{2/\eta}\rfloor  \big)$ for some integer $m \geq 0$, that is, $N \in [N_{k_{0}(\alpha)+m}, N_{k_{0}(\alpha)+m+1})$. Since $H_N$ increases in $N$, we have 
$H_{N_{k_{0}(\alpha)+m}} \leq H_N \leq H_{N_{k_{0}(\alpha)+m+1}}$. Applying Lemma~\ref{relation between H_N_k}, we obtain for all $N \geq \lfloor k_{0}(\alpha)^{2/\eta }\rfloor$,
\begin{align}
    \lfloor \delta_{\min}^{\alpha} \rfloor(N)  \leq  \lfloor\delta_{\min}^{\alpha} \rfloor(N_{k_{0}(\alpha)+m}) 
    \leq  \frac{N_{k_{0}(\alpha)+m}^\epsilon}{H_{N_{k_{0}(\alpha)+m}}}
    \leq  \frac{9N^\epsilon}{H_{N_{k_{0}(\alpha)+m+1}}}
    \leq \frac{9N^\epsilon}{H_{N}}
\end{align}
for any $\epsilon >0$.

Let $b >0$. Then, for all $N\geq \max (\lfloor k_{0}(\alpha)^{2/\eta }\rfloor,9^{2/b}),$
\begin{align}
     \lfloor \delta_{\min}^{\alpha} \rfloor(N)  
    \leq \frac{9N^{b/2}}{H_{N}}\leq \frac{N^{b}}{H_{N}}
\end{align}
for any $b >0$.

\section{Proof of Theorem~\ref{THM 3}}\label{Proof of theorem three}
We now provide a proof for Theorem~\ref{THM 3}. The following statement, widely known as Catlin's conjecture \cite{CATLIN1976}, was recently resolved as a consequence of Koukoulopoulos and Maynard's proof of the Duffin–Schaeffer conjecture \cite[Theorem 2]{KoukouMaynard}.

\begin{lemma}\label{Catlin Conjecture}
    Let $(\psi(k))_{k \geq 1}$ be a sequence of non-negative reals. Let $\mathcal{B}$ denote the set of those $\alpha \in[0,1]$ for which the inequality
$$
\left|\alpha-\frac{a}{k}\right| \leq \frac{\psi(k)}{k}
$$
has infinitely many solutions $(a, k)$ with $0 \leq a \leq k$. Then $\lambda(\mathcal{B})=0$ or $\lambda(\mathcal{B})=1$, according to whether the series
\begin{align}
    \sum_{k=1}^{\infty} \varphi(k) \sup _{b \geq 1}\left\{\frac{\psi(b k)}{b k}\right\}
\end{align}
is convergent or divergent, respectively.
\end{lemma}

\begin{lemma}\label{Satz 21}
Let $(\lambda_n)_{n \geq 1}$ be an arbitrary sequence of real numbers. Suppose that as a function of the index, \(\lambda\) satisfies the following condition: There exist two positive constants \(\epsilon\) and \(c\) such that whenever the index increases from \(n\) by more than \(\frac{n}{(\log n)^{1+\epsilon}}\), the value of \(\lambda\) increases by at least \(c\).  Then, the sequence $(\alpha \lambda_n)_{n \geq 1}$ is uniformly distributed for almost all $\alpha \in [0,1)$.
\end{lemma}
\begin{proof}
    See \cite[Satz 21]{Weyl1916}.
\end{proof}
\begin{lemma}\label{Weyl equidistribution}
    Let $(\lambda_n)_{n \geq 1}$ be a sequence of distinct integers. Then, the sequence $(\alpha \lambda_n)_{n \geq 1}$ is uniformly distributed mod $1$ for almost all real numbers $\alpha $.
\end{lemma}
\begin{proof}
    See \cite[Theorem 4.1]{kuipers1974uniform}.
\end{proof}
\begin{lemma}\label{Divergence lemma}
    Let $(a_n)_{n \geq 1}$ be a sequence of non-negative real numbers and $A$ be an infinite subset of $\mathbb{N}$ such that for each $k \in A$, $\displaystyle\sum_{\substack{d|k\\d\geq \log{k}}} a_d \geq \epsilon$ for some fixed $\epsilon >0$. Then $\displaystyle\sum_{k=1}^{\infty}a_{k}=\infty$.
\end{lemma}
\begin{proof}
    For each positive integer $N$, let $T(N) \coloneqq \#  \{d: d \mid N ,\,d \geq \log N\}$, and we enumerate the divisors as $$ \{d: d \mid N ,\,d \geq \log N\}=\{d_1^N, d_2^N, \dots, d_{T(N)}^N\},$$ where $d_{1}^{N}<d_{2}^{N}<\ldots <d_{T(N)}^{N}$.
    Since $A$ is an infinite set, there exists $k_1 \in A$ such that $\displaystyle\sum_{\substack{d\mid k_1 \\d\geq \log{k_1}}} a_d \geq \epsilon$, that is,
    \begin{align}
        a_{d_1^{k_1}}+a_{d_2^{k_1}}+\ldots+a_{d_{T(k_1)}^{k_1}}\geq \epsilon.
    \end{align}
    Now, choose $k_2 \in A$ such that $k_2 \geq e^{k_1}$ (this is possible since $A$ is unbounded). By hypothesis, we also have $\displaystyle\sum_{\substack{d\mid k_2 \\d\geq \log{k_2}}} a_d \geq \epsilon$, that is, 
    \begin{align}
        a_{d_1^{k_2}}+a_{d_2^{k_2}}+\ldots+a_{d_{T(k_2)}^{k_2}}\geq \epsilon.
    \end{align}
   Now, for each $d\mid k_2$ satisfying $d\geq \log k_2$, we have $d\geq k_1$, that is, $d > d_{i}^{k_1}$ for each $i=1,\ldots,T(k_1)$. Thus, $d_{j}^{k_2}> d_{i}^{k_1} $ for all $i \in \{ 1,\ldots,T(k_1)\}$ and $j \in \{ 1,\ldots,T(k_2)\}$.

   By induction, construct a sequence $(k_n)_{n \geq 1}$ in $A$ such that $k_n \geq e^{k_{n-1}}$ and $\displaystyle\sum_{\substack{d \mid k_n \\d\geq \log{k_n}}} a_d \geq \epsilon$, that is,
   \begin{align}
       a_{d_1^{k_n}}+a_{d_2^{k_n}}+\ldots+a_{d_{T(k_n)}^{k_n}}\geq \epsilon.
   \end{align}
By construction, we have $d_{i}^{k_{m}} < d_{j}^{k_n}$ for all $m<n$, $i \in \{ 1,\ldots,T(k_{m})\}$, and $j \in \{ 1,\ldots,T(k_{n})\}$. Also, for fixed $N$, $d_r^N < d_s^N$ for $1 \leq r < s \leq T(N)$. Therefore, 
\begin{align}
    & \displaystyle\sum_{k=1}^{d_{T(k_n)}^{k_n}}a_{k}\geq  (a_{d_1^{k_1}}+\ldots+a_{d_{T(k_1)}^{k_1}})+(a_{d_1^{k_2}}+\ldots+a_{d_{T(k_2)}^{k_2}})+\ldots+(a_{d_1^{k_n}}+\ldots+a_{d_{T(k_n)}^{k_n}})
    \geq  n\epsilon
\end{align}
for each $n \in \mathbb{N}$. Hence, $\displaystyle\sum_{k=1}^{\infty}a_k=\infty$.
\end{proof}
\begin{cor}\label{Divergence lemma corollary}
    Let $(a_n)_{n \geq 1}$ be a sequence of non-negative real numbers and $A$ be an infinite subset of $\mathbb{N}$ such that for all sufficiently large $k \in A$, $\displaystyle\sum_{\substack{d|k\\d\geq \log{k}}} a_d \geq \epsilon$ for some fixed $\epsilon >0$. Then $\displaystyle\sum_{k=1}^{\infty}a_{k}=\infty$.
\end{cor}
\begin{proof}
The result follows from Lemma~\ref{Divergence lemma} with a minor adjustment to the construction. Since the hypothesis applies to all sufficiently large $k \in A$, we choose $k_1$ large enough to satisfy the condition. For $n \geq 2$, we select $k_n \in A$ such that $k_n \geq e^{k_{n-1}}$. This ensures that all $k_n$ are sufficiently large, as $k_n > k_{n-1}$ for all $n$. The remainder of the proof proceeds identically.
\end{proof}
\begin{lemma}\label{Euler phi bound}
Let $k \in \mathbb{N}$ and let $\varphi$ be Euler's totient function.
\begin{enumerate}
    \item $\displaystyle\sum_{d \mid k} \varphi(d) = k$.
    \item For all sufficiently large $k$, that is, for $k \gg 1$, 
$ \displaystyle\sum_{\substack{d \mid k \\ d < \log k}} \varphi(d) < \frac{k}{2}$.
\end{enumerate}
\end{lemma}
\begin{proof}
\leavevmode 
We only prove $(2)$, as $(1)$ is a standard fact in number theory.
Since $\varphi(d) \leq d$ for all $d \in \mathbb{N}$, we have
\begin{align}
\sum_{\substack{d \mid k \\ d < \log k}} \varphi(d) 
\leq \sum_{\substack{d \mid k \\ d < \log k}} d 
\leq \sum_{\substack{d \mid k \\ d < \log k}} \log k
\leq (\log k)^2< \frac{k}{2}.
\end{align}
The last inequality holds for all sufficiently large $k$, which concludes the proof.
\end{proof}
\begin{lemma}\label{lim sup positive}
    Let $(a_{n})_{n \geq 1}$ be a sequence of positive real numbers such that $\displaystyle\inf_{1\leq m \neq n\leq N}|a_{m}-a_{n}|\geq 1$, and $(\eta(N))_{N \geq 1}$ be a sequence of non-negative reals such that \( \displaystyle\limsup_{k \to \infty} \eta(k) > 0 \). Then, the inequality \( \| \alpha\lfloor |a_m - a_n| \rfloor  \| \leq \eta(N) \) has infinitely many solutions for almost all \( \alpha \in [0,1] \).
\end{lemma}
\begin{proof}
    Let $\displaystyle\limsup_{N \to \infty} \eta(N) = u^* > 0$. Then, there exist infinitely many $N$ such that $\eta(N) > \frac{u^*}{2}$.

Let $A = \displaystyle\bigcup_{N \geq 1} \lfloor(A_N - A_N)^+ \rfloor $. The spacing condition and Lemma~\ref{H_N bound for general sequence} ensure that
\begin{align}
\# A \geq \# \lfloor (A_{N+1} - A_{N+1})^{+} \rfloor = H_{N+1} \geq \frac{N}{2}
\end{align}
for each \( N \in \mathbb{N} \). This shows that \( A \) is an infinite set of positive integers.

    Therefore, $A$ can be enumerated as $\{Z_1, Z_2, \dots\}$, ensuring that for each $N$,
    \begin{align}
            \{Z_n : 1 \leq n \leq H_N\} = \lfloor(A_N - A_N)^+ \rfloor.
    \end{align}
    Since $(Z_n)_{n \geq 1}$ is a sequence of distinct integers, Weyl's equidistribution theorem (see Lemma~\ref{Weyl equidistribution}) implies that for almost every $\alpha $, the sequence $(\alpha Z_n)_{n \geq 1}$ is uniformly distributed mod $1$. Hence, for any $\delta > 0$,
    \begin{align}
    \lim_{N \to \infty} \frac{1}{N} \#\left\{k \leq N : \{\alpha Z_k\} \in [0, \delta]\right\} = \delta.
    \end{align}
    In particular, taking $\delta = \frac{u^*}{2}$, for almost every $\alpha$, there exist infinitely many $k$ such that $\{\alpha Z_k\} \leq \frac{u^*}{2} $, that is, $\|\alpha Z_k\| \leq \frac{u^*}{2} $.

  Since $\eta(N) > \frac{u^*}{2}$ for infinitely many $N$, there exist sequences of distinct integers $(k(t))_{t \geq 1}$ and $(N(s))_{s \geq 1}$ such that
    \begin{align}
    \|\alpha Z_{k(t)}\| \leq \frac{u^*}{2}<\eta(N(s))
    \end{align}
    for almost every $\alpha$.  Since $(N(s))_{s\geq 1}$ is unbounded, there is a sequence of distinct positive integers $(r(t))_{t \geq 1}$ such that $ N(r(t))> 2k(t)+1$.
    By Lemma~\ref{H_N bound for general sequence}, $ H_{N(r(t))} \geq \frac{N(r(t))-1}{2}>k(t)$. Therefore, for each $t\in \mathbb{N}$,
    \begin{align}
        \delta^{\alpha}_{\min}(N(r(t))) = \min_{1 \leq m \ne n \leq N(r(t))} \|\alpha \lfloor |a_m - a_n| \rfloor\| = \min_{l \leq H_{N(r(t))}} \|\alpha Z_l\| \leq \|\alpha Z_{k(t)}\| \leq \frac{u^*}{2}<\eta(N(r(t))).
    \end{align}
    Since $(r(t))_{t \geq 1}$ and $(N(s))_{s\geq 1}$ are sequences of distinct integers, we have $(N(r(t))_{t \geq 1}$ is a sequence of distinct integers. Therefore, for almost every $\alpha$, there are infinitely many $N$ such that
    \begin{align}
    \delta^{\alpha}_{\min}(N) \leq \eta(N).
    \end{align}
\end{proof}
\begin{remark}
    Since a sequence of distinct integers has an increasing subsequence, by passing to one such subsequence $(Z_{l(n)})_{n \geq 1}$ of $(Z_n)_{n \geq 1}$ and applying Lemma~\ref{Satz 21}, we find that $(\alpha Z_{l(n)})_{n \geq 1}$ is uniformly distributed for almost all $\alpha \in [0,1)$. Hence, Lemma~\ref{Satz 21} can be used in place of Lemma~\ref{Weyl equidistribution} in the proof of Lemma~\ref{lim sup positive}, as noted in \cite{ABM}.
    \end{remark}    
\subsection*{Proof of Theorem~\ref{THM 3}}  We break the proof in two cases, depending on whether $\displaystyle\limsup _{k \rightarrow \infty} \eta(k)>0$ or equals $0$, where $(\eta(k))_{k \geq 1}$ is the given sequence of non-negative reals. 
    
\textbf{Case 1:} Let \( \displaystyle\limsup_{k \to \infty} \eta(k) =u^{*}> 0 \). We will show that the series \eqref{series conv div} diverges and that \( \lambda(\mathcal{A}) = 1 \). By Lemma~\ref{lim sup positive}, inequality \( \| \alpha\lfloor |a_m - a_n| \rfloor  \| \leq \eta(N) \) has infinitely many solutions for almost all \( \alpha \in [0,1] \). Consequently, \( \lambda(\mathcal{A}) = 1 \).

    We now show that the series \eqref{series conv div} is indeed divergent in this case. Let $A=\displaystyle\bigcup_{N \geq 1}\lfloor\left(A_{N}-A_{N}\right)^{+}\rfloor $ and fix a number $k \in A$. Then we have $\mathcal{N}(k)<\infty$, where $\mathcal{N}(k)$ is defined in Section~\ref{Notation}. An application of Lemma~\ref{Euler phi bound} yields
\begin{align}\label{series ad}
\sum_{\substack{d \mid k\\ d \geq \log k}} \varphi(d) \sup _{b \geq 1}\left\{\frac{\sup _{\ell \geq \mathcal{N}(b d)} \eta(\ell)}{b d}\right\} & \geq \sum_{\substack{d \mid k\\ d \geq \log k}} \varphi(d) \frac{\sup _{\ell \geq \mathcal{N}(k)} \eta(\ell)}{k} \\
& =\frac{\sup _{\ell \geq \mathcal{N}(k)} \eta(\ell)}{k}\left(\sum_{d \mid k} \varphi(d)-\sum_{\substack{d \mid k\\ d<\log k}} \varphi(d)\right) \\
& \geq \frac{\sup _{\ell \geq \mathcal{N}(k)} \eta(\ell)}{2} \geq \frac{u^{*}}{2}>0,
\end{align}
for sufficiently large $k \in A$, where $\mathcal{N}(k)<\infty$ ensures that the supremum is finite (condition $\mathcal{N}(k)<\infty$ ensures that we are not taking a supremum over an empty set, in which case it would be $0$, by convention). Since $A$ is an infinite set, we use \eqref{series ad} and Corollary~\ref{Divergence lemma corollary} to deduce that \eqref{series conv div} is indeed divergent in this case.

\textbf{Case 2:} Let \(\displaystyle \limsup_{k \to \infty} \eta(k) = 0 \). Since $(\eta(k))_{k \geq 1}$ is a sequence of non-negative reals, we have $\eta(N) \rightarrow 0$. We set
$$\psi(k)=\begin{cases}
    \displaystyle\sup _{\ell \geq \mathcal{N}(k)} \eta(\ell)  &\quad\text{if }k \in A, \\        0 &\quad\text{otherwise.} \\ 
\end{cases}$$
For every $k \in A$, we set
\begin{align}
S_{k}=[0,1] \cap\left(\bigcup_{0 \leq a \leq k}\left(\frac{a}{k}-\frac{\psi(k)}{k}, \frac{a}{k}+\frac{\psi(k)}{k}\right)\right),
\end{align}
and set $S_{k}=\emptyset$ for $k \notin A$. Let $$\mathcal{A}=\{ \alpha \in [0,1]: \lfloor \delta_{\min}^{\alpha}\rfloor (N) \leq \eta(N) \text{ for infinitely many }N\}$$
be defined as in the statement of the theorem, and $$\mathcal{B}=\{ \alpha \in [0,1]:  \alpha \in S_k\text{ for infinitely many }k\}.$$

To complete the proof of Theorem~\ref{THM 3}, we will show that $\mathcal{A}=\mathcal{B}$.
By Lemma~\ref{Catlin Conjecture}, we have $\lambda(\mathcal{B})=0$ or $\lambda(\mathcal{B})=1$, according to whether the series
\begin{align}\label{two series}
   \sum_{k=1}^{\infty} \varphi(k) \sup _{b \geq 1}\left\{\frac{\psi(b k)}{b k}\right\}=\sum_{k=1}^{\infty} \varphi(k) \sup _{b \geq 1}\left\{\frac{\sup _{\ell \geq \mathcal{N}(b k)} \eta(\ell)}{b k}\right\} 
\end{align}
converges or diverges, respectively. Note that $\mathcal{N}(k)=\infty$ for $k \notin A$, and the supremum over the empty set is zero by convention. Since $\eta(\ell) \geq 0$, both expressions 
\begin{align}
    \displaystyle\sup _{b \geq 1}\left\{\frac{\psi(b k)}{b k}\right\} \quad \text{and} \quad \displaystyle\sup _{b \geq 1}\left\{\frac{\sup _{\ell \geq \mathcal{N}(b k)} \eta(\ell)}{b k}\right\} 
\end{align}
coincide with $\displaystyle\sup_{\substack{b \geq 1\\bk \in A}} \left( \frac{\sup_{\ell \geq \mathcal{N}(bk)}\eta(\ell)}{bk}   \right) $ for all $k \in \mathbb{N}$. Therefore, the two series in \eqref{two series} are equal.

To show $\mathcal{A} = \mathcal{B}$, first assume that $\alpha \in \mathcal{A}$. Then, there are infinitely many values of $N, m, n$ with $1 \leq m \neq n \leq N$ such that
$    \| \alpha \lfloor |a_{m}-a_{n}| \rfloor \| \leq \eta(N)$.
Set $k=\lfloor |a_{m}-a_{n}| \rfloor$. Since $m, n \leq N$, we have $\mathcal{N}(k) \leq N$. Thus, we have $\psi(k) = \sup_{\ell \geq \mathcal{N}(k)} \eta(\ell) \geq \eta(N)$.
Consequently, there are infinitely many values of $k, N $ such that $ \|k \alpha\| \leq \eta(N)$, that is,  $\|k \alpha\| \leq \psi(k)$, and thus $\alpha \in S_{k}$. Since $\eta(N) \rightarrow 0$, a particular difference $k=\lfloor |a_{m}-a_{n}|\rfloor$ can generate only finitely many values $N$ (together with $m, n \leq N)$ such that $   \| \lfloor |a_{m}-a_{n}| \rfloor \alpha\| \leq \eta(N)$. Since each $k$ corresponds to finitely many $N$, and there are infinitely many values of $k, N$ such that $ \|k \alpha\| \leq \eta(N)$, the choices for $k$ satisfying $ \|k \alpha\| \leq \eta(N)$ must be infinite. In other words, there are infinitely many distinct values of $k$ such that $\alpha \in S_k$, and therefore $\mathcal{A} \subseteq \mathcal{B}$.

We now assume that $\alpha \notin \mathcal{A}$. Then, there are only finitely many $N$ such that $\lfloor \delta_{\min}^{\alpha}\rfloor (N) \leq \eta(N)$, that is, there are only finitely many $N$ and $m, n$ with $1 \leq m \neq n \leq N$ such that $    \| \alpha \lfloor |a_{m}-a_{n}| \rfloor \|\leq \eta(N)$. This implies that there are only finitely many $k, N  \in \mathbb{N}$ such that $k \in \lfloor ( A_{N}-A_{N})^{+}\rfloor $ and $\| k\alpha \| \leq \eta(N)$. Since $\eta(N) \rightarrow 0$, there are only finitely many $k$, and for each such $k$, there are only finitely many $\ell \geq \mathcal{N}(k)$ satisfying $\|k \alpha \| \leq \eta (l)$. Therefore, there are only finitely many $k$ for which $\|k \alpha\| \leq \sup _{l \geq \mathcal{N}(k)}\eta (l)$, that is, $\| k \alpha \| \leq \psi(k)$. Thus, there are only finitely many $k$ such that $\alpha \in S_k$, which shows that $\alpha \notin \mathcal{B}$.

Thus, we have shown that $\mathcal{A}=\mathcal{B}$. The convergence or divergence of \eqref{series conv div} provides a zero-one law for $\lambda(\mathcal{B})$, and hence also for $\lambda(\mathcal{A})$, as stated in Theorem~\ref{THM 3}.

\subsection*{Acknowledgements}
The author is grateful to Christoph Aistleitner and Sudhir Pujahari for thoughtful inputs and correspondence. The author is supported by the Department of Atomic Energy (DAE project number RIN-4001).

\bibliographystyle{plain}
\bibliography{minimal}
\end{document}